\documentclass[11pt]{amsart}
\usepackage{amsmath}
\usepackage{amssymb}
\usepackage{amsthm}
\usepackage{amscd}
\usepackage{amsfonts}
\usepackage{graphicx}
\usepackage{fancyhdr}
\usepackage{epsfig}
\usepackage{tikz}
\usepackage{setspace}
\usepackage{verbatim}
\usepackage{subfigure}
\usepackage{cleveref}
\usepackage{tikz}
\usepackage{enumitem}

\usetikzlibrary{patterns}
\usepackage{amsmath, times}
\theoremstyle{plain} \numberwithin{equation}{section}
\newtheorem{theorem}{Theorem}[section]

\newtheorem{lemma}[theorem]{Lemma}
\newtheorem{proposition}[theorem]{Proposition}

\theoremstyle{definition}
\newtheorem{definition}[theorem]{Definition}

\newtheorem{remark}[theorem]{Remark}
\newtheorem{example}[theorem]{Example}
\newtheorem{notation}[theorem]{Notation}

\DeclareMathOperator{\lcm}{lcm}

\DeclareMathOperator{\pd}{pd}
\DeclareMathOperator{\reg}{reg}

\DeclareMathOperator{\tor}{Tor }

\title{Projective dimension of hypergraphs }

\author[Lin]{Kuei-Nuan Lin}
\address{Department of Mathematics, The Penn State University, Greater Allegheny Campus,  McKeesport, PA}
\email{kul20@psu.edu}

\author[Mapes]{Sonja Mapes}
\address{Department of Mathematics, University of Notre Dame,  Notre Dame, IN}
\email{smapes1@nd.edu}


\date{\today}

\begin{document}
\maketitle

\begin{abstract}
Given a square-free monomial ideal $I$, satisfying certain hypotheses, in a polynomial ring $R$ over a field $\mathbb{K}$, we compute the projective dimension of $I$. Specifically, we focus on the cases where the 1-skeleton of an associated hypergraph is either a string or a cycle. We investigate the impact on the projective dimension  when higher dimensional edges are removed. We prove that the higher dimensional edge either has no effect on the projective dimension or the projective dimension only goes up by one with the extra higher dimensional edge. 
\end{abstract}

\section{Introduction}

Let $R = \mathbb{K}[x_1, \dots, x_n]$ be a polynomial ring over a field $\mathbb{K}$.  The minimal free resolution of $R/I$ for an ideal $I\subset R$ is an exact sequence of the form 
	\[
	0  \rightarrow  \bigoplus_j S(-j)^{\beta_{p,j}(R/I)}  \rightarrow \dots  \rightarrow  \bigoplus_j S(-j)^{\beta_{1,j}(R/I)}  \rightarrow  R \rightarrow  R/I  \rightarrow  0
		\]
The exponents $\beta_{i,j}(R/I)$ are invariants of $R/I$, called the Betti numbers of $R/I$. In general, finding Betti numbers is still a wide open question. Two other invariants that one can associate to a minimal resolution are the projective dimension of $R/I$, denoted $\pd(R/I)$, which is defined as follows
	\begin{align*}
		\pd(R/I) &= \max\{ i \mid \beta_{i,j}(R/I) \neq 0\},
	\end{align*}
and the (Castelnuovo-Mumford) regularity of $R/I$, denoted reg($R/I$), which is defined as follows
	\begin{align*}
		\mbox{reg}\,(R/I) &= \max\{ j-i \mid \beta_{i,j}(R/I) \neq 0\}.
	\end{align*}
 
 Those two invariants play important roles in algebraic geometry, commutative
algebra, and combinatorial algebra. In general, one finds the graded
minimal free resolution of an ideal to obtain those invariants, but
the computation can be difficult and computationally
expensive.  

Kimura, Terai and Yoshida define the dual hypergraph of a square-free monomial ideal in order to compute
its arithmetical rank \cite{KTY}(see \Cref{HGDef} for definition of a hypergraph). Since then, there have been a couple of papers using this combinatorial object to
study various properties, for example, \cite{HaLin} and \cite{LMc}. In
particular, Lin and Mantero use it to show that ideals with the same dual hypergraph
have the same Betti numbers and projective dimension \cite{LMa1} (\Cref{LaundryList1} \ref{bettiNumbersSame}), which has found use in
other papers, such as in \cite{KMa}. 

The focus of this work is to use the properties of a hypergraph to compute the projective dimension of the associted square-free monomial ideal without finding the minimal free resolution of the ideal. This is a different focus than various work by others which produce the full resolution, for example, the recent work of Eagon, Miller, and Ordog \cite{EMO}. More precisely, we find the projective dimension of hypergraphs when their 1-skeleton is a string or cycle. This extends the work of Lin and Mantero in \cite{LMa1}. A key development is the result of Lin and Mapes in \cite{LMap}, which allows us to remove a large class of higher dimensional edges of a hypergraph without impacting its projective dimension (Corollary 4.4 \cite{LMap}).  Given this previous result we can restrict our study in this paper to specific cases which are not previously dealt with.  Our main results are summarized by the following theorem. 

\begin{theorem}
Let $\mathcal{H}$ be a hypergraph and $F$ be an edge of $\mathcal{H}$ such that $\mathrm{dim}F>0$.
\begin{enumerate}
\item Let $\mathcal{H}_S$ be an open string of length $\mu$. If $\mathcal{H}$ is the union of $\mathcal{H}_S$ and $F$ such that all vertices of $F$ are on $\mathcal{H}_S$, then $\pd(\mathcal{H})=\pd(\mathcal{H}_S)=\mu-\left\lfloor \frac{\mu}{3}\right\rfloor$ or $\pd(\mathcal{H})=\pd(\mathcal{H}_S)+1$ depending on the position of vertices of $F$.
\item Let $\mathcal{H}_{C_\mu}$ be an open cycle of length $\mu$. If $\mathcal{H}$ is the union of $\mathcal{H}_{C_\mu}$ and $F$ such that all vertices of $F$ are on $\mathcal{H}_{C_\mu}$, then $\pd(\mathcal{H})=\pd(\mathcal{H}_{C_\mu})=\mu-1-\left\lfloor \frac{\mu-2}{3}\right\rfloor$.
\end{enumerate}
\end{theorem}

In section 2 of the paper we begin by giving the necessary background to understand the results of this paper.  We also establish a new technique for computing the projective dimension using bounds on sub-ideals, which is inspired by methods
from \cite{fatabbi}, namely Betti splittings (\Cref{Split}).  We then proceed with our
results concerning higher dimensional edges on strings and  cycles in \Cref{StringsSec}, and \Cref{CyclesSec}. Through out this paper, ideals are monomial ideals in a polynomial ring $R$ over the field $\mathbb{K}$.

\section{Preliminaries}\label{prelimSec}

\subsection{Hypergraph of a square-free monomial ideal}

Kimura, Terai, and Yoshida associate a square-free monomial ideal with
a hypergraph in \cite{KTY}, see Definition \ref{HGDef}.  Note that
this construction is different from the constructions associating ideals
to hypergraphs coming from the study of edge ideals.  In particular
relative to edge ideals, the hypergraph of Kimura, Terai, and Yoshida
might be more aptly named the ``dual hypergraph''.   The construction
of the dual hypergraph is first introduced by Berge in \cite{Berge}. In
the edge ideal case, one associates a square-free monomial with a
hypergraph by setting variables as vertices and each monomial
corresponds to an edge of the hypergraph (see for example
\cite{Ha}). In the following definition, we actually associate
variables with edges of the hypergraph and vertices with the monomial generators
of the ideal, and in practice this is the dual hypergraph of the
hypergraph in the edge ideal construction.

\begin{definition}\label{HGDef}
Let $I$ be a square-free monomial ideal in a polynomial ring with $n$ variables with minimal monomial
generating set $\{m_1, \dots, m_\mu\}$.  Let $V$ be the set
$\{1,\dots, \mu\}$.  We define $\mathcal{H}(I)$ (or $\mathcal{H}$ when
$I$ is understood) to
be the hypergraph associated to $I$ which is defined as $\{\{j \in V :
x_i|m_j\} : i = 1,2,\dots, n\}$.  We call the sets $\{j \in V :
x_i|m_j\}$ the \emph{edges} of the hypergraph. 
\end{definition}

Note that if you start with a hypergraph you can create a monomial ideal by assigning a variable to each edge, then each vertex (or element in $V$) would be assigned the monomial product of the variables corresponding the the edges using that vertex.  The issue however that doing this will not always produce a minimal generating set. To obtain a minimal generating set the hypergraph needs to be separated, where $\mathcal{H}$ is
\emph{separated} if in addition for every $1 \leq j_1 < j_2 \leq
\mu$, there exist edges $F_1$ and $F_2$ in $\mathcal{H}$ so that $j_1
\in F_1 \cap (V-F_2)$ and  $j_2
\in F_2 \cap (V-F_1)$.  All hypergraphs in this paper will be separated unless otherwise stated. 

\begin{example} \label{bigStringEx}
Let \begin{align*}
    I=&(m_1=abk,m_2=bcl,m_3=cdklm,m_4=dekn,\\& m_5=efgn,m_6=ghmn,m_7=hikl,m_8=ijk)
\end{align*} the \Cref{FirstHG} is the hypergraph associated to $I$ via the \Cref{HGDef} where
\begin{align*}&\mathcal{H}(I)=\{a=\{1\},b=\{1,2\},c=\{2,3\},d=\{3,4\},e=\{4,5\},f=\{5\},g=\{5,6\},h=\{6,7\},\\
&i=\{7,8\},j=\{8\},k=\{1,3,4,7,8\},l=\{2,3,7\},m=\{3,6\},n=\{4,5,6\}\}.
\end{align*}

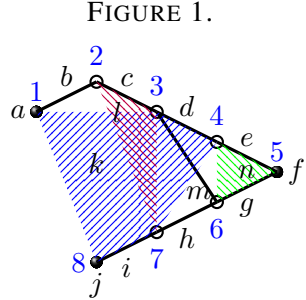
\begin{figure}[h] 
\caption{}\label{FirstHG}
\begin{center}
\begin{tikzpicture}[thick, scale=0.8]

\shade [shading=ball, ball color=black]  (3,-1.5) circle (.1) node [left] {$\textcolor{blue}8$} node [below] {$j$};
\draw  [shape=circle] (4,-1) circle (.1) node [below] {$\textcolor{blue}7$};
\draw  [shape=circle] (5,-0.5) circle (.1) node [below] {$\textcolor{blue}6$};
\shade [shading=ball, ball color=black] (6,0) circle (.1) node [above] {$\textcolor{blue}5$} node [right] {$f$};
\draw  [shape=circle] (5,0.5) circle (.1) node [above] {$\textcolor{blue}4$};
\draw  [shape=circle] (4,1) circle (.1)  node [above] {$\textcolor{blue}3$};
\draw  [shape=circle] (3,1.5) circle (.1) node [above] {$\textcolor{blue}2$};
\shade [shading=ball, ball color=black]  (2,1) circle (.1)  node [left] {$a$}  node [above] {$\textcolor{blue}1$};

\draw [line width=1.2pt] (3,-1.5)--(4,-1) node [pos=.5, below] {$i$};
\draw [line width=1.2pt ] (5,-0.5)--(4,1) node [pos=.3, below] {$m$};
\draw [line width=1.2pt ] (4,-1)--(5,-0.5) node [pos=.5, below] {$h$};
\draw [line width=1.2pt] (5,-0.5)--(6,0) node [pos=.5, below] {$g$};
\draw [line width=1.2pt] (6,0)--(5,0.5) node [pos=.5, above] {$e$};
\draw [line width=1.2pt] (5,0.5)--(4,1) node [pos=.5, above] {$d$};
\draw [line width=1.2pt] (4,1)--(3,1.5) node [pos=.5, above] {$c$};
\draw [line width=1.2pt] (3,1.5)--(2,1) node [pos=.5, above] {$b$};
\path [pattern=north east lines, pattern  color=blue]   (3,-1.5)--(4,-1)--(5,0.5)--(4,1)--(2,1)--cycle;
\path [pattern=north west lines, pattern  color=green]   (6,0)--(5,-0.5)--(5,0.5)--cycle;
\path [pattern=north west lines, pattern  color=purple]   (4,-1)--(3,1.5)--(4,1)--cycle;

\path (3,0.5)--(3,0.5) node [pos=.5, below ] {$k$ };
\path (3.2,1.35)--(3.5,1.35) node [pos=.5, below ] {$l$ };
\path (5,0)--(6,0) node [pos=.5 ] {$n$ };

\end{tikzpicture}
\end{center}
\end{figure}
\end{example}

Some important terminology regarding these hypergraphs is the
following.  We say a vertex $i \in V$ of $\mathcal{H}$ is an
\emph{open vertex} if $\{i\}$ is not in $\mathcal{H}$, and otherwise
$i$ is \emph{closed}.  In Figure \ref{FirstHG}, we can see that the
vertices labeled by $a, f$ and $j$ are all closed, and the rest
are open.   Moreover, a hypergraph $\mathcal{H}$ with $V=[\mu]$ is a
\emph{string} if $\{i, i+1\}$ is in $\mathcal{H}$ for all $i = 1,
\dots, \mu-1$, and the only edges containing $i$ are $\{i-1, i\}$,
$\{i, i+1\}$ and possibly $\{i\}$.  We say that a string is an
  \emph{open string} if all vertices other than $1$ and $\mu$ are
  open (Note that for $\mathcal{H}$ to be separated $1$ and $\mu$ must be closed).  Also,
  $\mathcal{H}$ is a \emph{$\mu$-cycle} if $\mathcal{H} = \tilde{\mathcal{H}}
  \cup \{\mu,1\}$ where $\tilde{\mathcal{H}}$ is a string.  We say a
  cycle is an \emph{open cycle} if all the vertices are open. In Figure \ref{FirstHG} if we were to remove the edges corresponding to the variables $k, l, m,$ and $n$ then we would be left with a string.  If we further removed the edge corresponding to the variable $f$, then it would be an open string. Let  $\mathcal{H}^{i} = \{ F \in H : |F| \leq i +1\}$ denote the $i$-th dimensional subhypergraph of $\mathcal{H}$ where $|F|$ is the cardinality of the $F$. We call $\mathcal{H}^{1}$ the 1-skeleton of $\mathcal{H}$.

  \begin{example}
  Some simple examples where $V = \{1,2,3\}$ of open strings and open cycles are as follows.  Let $\mathcal{H}_1 = \{a = \{1\}, b= \{1,2\}, c=\{2,3\}, d=\{3\}\}$, then $\mathcal{H}_1$ is an open string and the corresponding monomial ideal is $I_1 = (ab, bc, cd)$. If we let $\mathcal{H}_2 = \{a=\{1,2\}, b=\{2,3\}, c=\{1,3\}\}$, then this is an open cycle and the corresponding monomial ideal is $I_2 = (ac,ab,bc) $.   
  \end{example}

Recently there have been a number of results determining both the
projective dimension and the regularity of certain square-free monomial ideals
from the associated hypergraph.  As this paper focuses more on the projective dimension, we include the statements of some of results that are useful for the rest of the paper here
(these appear separately in the literature but we list them all
here as part of one statement). We write $\pd(\mathcal{H})=\pd(\mathcal{H}(I))=\pd(R/I)$ and $\reg(\mathcal{H})=\reg(\mathcal{H}(I))=\reg(R/I)$ where $\mathcal{H}(I)$ is the hypergraph obtained from a square-free monomial ideal $I$. We do this for the Betti numbers as well. Notice that for the content of this work, we assume there is a unique variable associated to each face of $\mathcal{H}$.

\begin{theorem}\label{LaundryList1}
\hangindent\leftmargini
\hskip\labelsep
\begin{enumerate}[label={(\arabic*)},ref={(\arabic*)}]
\item \label{bettiNumbersSame} \cite[Proposition 2.2]{LMa1} If $I_1$
  and $I_2$ are square-free monomial ideals associated to the
same separated hypergraph $\mathcal{H}$, then the total Betti numbers of the two ideals coincide.
\item \label{removeMeets} \cite[Corollary 4.4]{LMap} Let $F$ be an edge on the hypergraph $\mathcal{H}$. If $F$ is an union of other edges of $\mathcal{H}$, then $\pd(\mathcal{H})=\pd(\mathcal{H}\backslash F)$. 
\item  \label{LMa2-2.9c} \cite[Theorem 2.9 (c)]{LMa2}) If
  $\mathcal{H}' \subseteq \mathcal{H}$ are hypergraphs with
  $\mu(\mathcal{H}') = \mu(\mathcal{H})$, then $\pd(\mathcal{H}')
  \leq \pd(\mathcal{H})$ where $\mu(\mathcal{*})$ denotes the number of vertices of $*$.
 \item  \label{pdFormulaString} \cite[Theorem 7.7.34, Corollary 7.7.35]{Jacques} An open string hypergraph with $\mu$ vertices has
projective dimension $\mu-\left\lfloor \frac{\mu}{3}\right\rfloor $, regularity $\left\lceil \frac{\mu}{3}\right\rceil $, and $\beta_{\mu-\left\lfloor \frac{\mu}{3}\right\rfloor ,\mu}(\mathcal{H})\neq0$ when ${\mu}(\mathcal{H})\neq0$.
\item \label{BettiString2}\cite[Corollary 2.2]{JK} A hypergraph $\mathcal{H}$ with $n_1+n_2$ vertices that is a disjoint union of open string hypergraphs, $\mathcal{H}_1$ and $\mathcal{H}_2$ with $n_1$ and $n_2$ vertices has $\pd(\mathcal{H})=n_1+n_2-\left\lfloor \frac{n_1}{3}\right\rfloor -\left\lfloor \frac{n_2}{3}\right\rfloor$ and $\beta_{n_1+n_2-\left\lfloor \frac{n_1}{3}\right\rfloor -\left\lfloor \frac{n_2}{3}\right\rfloor,n_1+n_2}(\mathcal{H})\neq0$.
\item \label{pdFormulaCycle} \cite[Corollary 7.6.30]{Jacques} If $\mathcal{H}$ is an open cycle with $\mu$ vertices then
$\pd({\mathcal{H}}) = \mu - 1- \lfloor \frac{\mu -2}{3} \rfloor$. 
\item \label{RegFormula}\cite[Corollary 20.19]{Eisenbud} If $I=mI'$ where $m$ is a monomial of degree
$r$ then $\pd(R/I)=\pd(R/I')$, $\reg(R/I)=r+\reg(R/I')$.
\item \label{RegFormula2} \cite[Corollary 20.19]{Eisenbud} Let $m$ be a monomial of degree
$r$, then  
$$\reg(R/(I,m))\leq\max\{\reg(R/I),\reg(R/(I:m)+r-1\}.$$
\end{enumerate}
\end{theorem}

\begin{remark}Note that Theorem \ref{LaundryList1} part \ref{bettiNumbersSame}
allows us to talk about the projective dimension of a
hypergraph rather than an ideal. We will use $\pd(\mathcal{H}(I)) $ in
the place of $ \pd(R/I)$ throughout the paper. Moreover if an edge is an union of other edges in a hypergraph, and we are considering the projective dimension of the hypergraph, we can just ignore or remove the edge using Theorem \ref{LaundryList1} part \ref{removeMeets}. For example, in Figure \ref{FirstHG}, we can remove the edge $k=\{1,3,4,7,8\}=\{1\}\cup\{3,4\}\cup\{7,8\}$ and $n=\{4,5,6\}=\{4,5\}\cup\{5,6\}$. If a hypergraph
$\mathcal{H}=\mathcal{H}(I)$ is an union of two disconnected hypergraphs $G_1=\mathcal{H}(I_1)$ and
$G_2=\mathcal{H}(I_2)$, we have $\pd(\mathcal{H})=\pd(G_1)+\pd(G_2)$ and $\reg(\mathcal{H})=\reg(G_1)+\reg(G_2)$ as one can construct the minimal resolution of $R/I$ using the tensor of minimal resolutions of $R/I_1$ and $R/I_2$.

\end{remark}

\subsection{Colon Ideals - Key tool}
One technique that is used in \cite{LMa1} and \cite{LMa2} which we will need
here, is using the short exact sequences obtained by looking at colon
ideals.  Specifically there are two types of colon ideals that we are
interested in, and we explain below what each operation looks like
on the associated hypergraphs.

\begin{definition}
Let $\mathcal{H}$ be a hypergraph, and $I = I(\mathcal{H})$ be the
standard square-free monomial ideal associated to it in the polynomial
ring $R$. Let $\mathcal{G}(I)=\{m_1,\dots, m_{\mu}\}$ be the minimal generating set of $I$. Let $F$ be an
edge in $\mathcal{H}$ and let $x_F \in R$ be the variable associated
to $F$.   Also let
$v$ be a vertex in $\mathcal{H}$ and $m_v \in I$ be the monomial
generator associated to it.
\begin{itemize}
\item The hypergraph $\mathcal{H}_v : v = \mathcal{Q}_v$ is the hypergraph associated to the ideal
$I_v : m_v$ where $I_v =(\mathcal{G}(I)\backslash m_{v}$), and
$\mathcal{H}_v = \mathcal{H}(I_v)$ is the hypergraph associated to the ideal $I_v$.  
\item The hypergraph $\mathcal{H} : F$, obtained by removing $F$
  in $\mathcal{H}$, is the hypergraph associated to the ideal $I : x_F$.
  \item The hypergraph $(\mathcal{H},x_F)$, obtained by adding a vertex corresponding to the variable $x_F$ in $\mathcal{H}$, is the hypergraph associated to the ideal $(I,x_F)$.

\end{itemize}
\end{definition}

\begin{remark}\label{shortexact}
Given a hypergraph $\mathcal{H}$ and a vertex $v$ of $\mathcal{H}$, we consider a short exact sequence 
\[
0\rightarrow \mathcal{H}_v : v = \mathcal{Q}_v\rightarrow\mathcal{H}_{v}\rightarrow\mathcal{H}\rightarrow0.
\]
If $\pd(\mathcal{H}_{v}) >\pd(\mathcal{Q}_{v})$,
then with the projective dimension on the short exact sequence that $\pd(\mathcal{H}_{v})\leq\text{max}\{\pd(\mathcal{Q}_{v}),\pd(\mathcal{H})\}$
and \[\pd(\mathcal{H})\leq\text{max}\{\pd(\mathcal{H}_{v}),\pd(\mathcal{Q}_{v})+1\},\]
we can conclude that $\pd(\mathcal{H})=\pd(\mathcal{H}_{v})$. We will use this strategy in later sections. 
\end{remark}
The following result appearing in \cite{LMa2}
will be very useful to us in this paper. We put it here for the self-containment of this work and for the reader's convenience. We need some terminology first. The degree of a vertex is the number of faces containing the vertex. Let $v$ be a vertex in a hypergraph $\mathcal{H}$, and let $\mathcal{H}_1, . . . , \mathcal{H}_r$ be the connected components of $\mathcal{H}_v$; if one of them, say $\mathcal{H}_1$, is
a string hypergraph, we call $\mathcal{H}_1$ a branch of $\mathcal{H}$ (from $v$).

\begin{theorem} \cite{LMa2} \label{LMa247} 
  Let $\mathcal{H}$ be a 1-dimensional hypergraph, $w$ a vertex with
  degree at least $3$ in $\mathcal{H}$, and $S$ be a branch departing
  from $w$ with $v_1,...,v_{n}$ vertices. Suppose vertices of $S$, $v_1,...,v_{n-1}$, are open and the end vertex $v_n$, i.e., the leaf of $S$, is the only closed vertex of $S$. Let $E$ be the unique edge connecting $w$ to $v_1$. Then $\pd(\mathcal{H}) = \pd(\mathcal{H}')$, where $\mathcal{H}'$ is the following hypergraph: (a) if $n\equiv1$ mod $3$, then $\mathcal{H}' =\mathcal{H}:E$; (b) if $n\equiv2$ mod $3$, then $\mathcal{H}'=\mathcal{H}_w$
\end{theorem}

\begin{example}
In \Cref{fig:H:E}, there are 4 hypergraphs, $\mathcal{H}$, $\mathcal{H}_1=\mathcal{H}:E$, $\mathcal{H}_2=\mathcal{H}_{w_2}$, $\mathcal{H}_4=(\mathcal{H},x_F)$. $S_1$ is a branch of length 1 departing from $w_1$ and $S_2$ is a branch of length 2 departing from $w_2$. By \Cref{LMa247}, $\pd(\mathcal{H}) = \pd(\mathcal{H}_1)=\pd(\mathcal{H}_2)$. Notice that in $\mathcal{H}_2$, the edge $F$ becomes an 1-dimensional edge. 

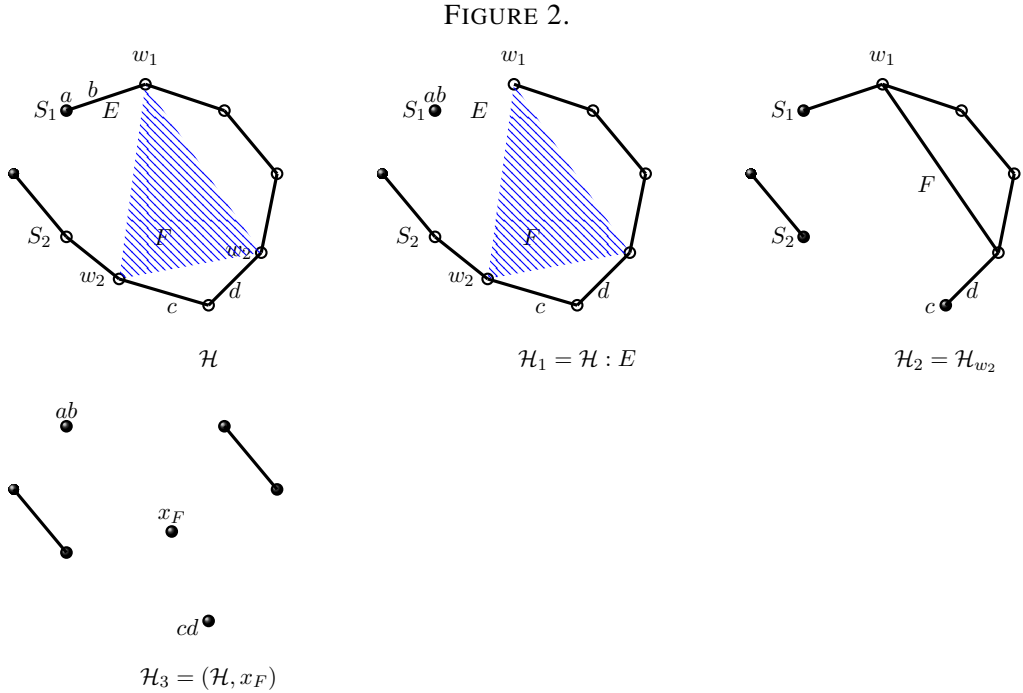
\begin{figure}[h] 
\caption{}\label{fig:H:E}
\begin{center}
\begin{tikzpicture}[thick, scale=0.7, every node/.style={scale=0.8}]
\shade [shading=ball, ball color=black]  (0,0) circle (.1);
\draw  [shape=circle] (1,-1.2) circle (.1)  node[left]{$S_2$};
\draw  [shape=circle] (2,-2) circle (.1)  node[left]{$w_2$};
\draw  [shape=circle] (3.7,-2.5) circle (.1);
\draw  [shape=circle] (4.7,-1.5) circle (.1);
\draw  [shape=circle] (5,0) circle (.1);
\draw  [shape=circle] (4,1.2) circle (.1) ;
\draw  [shape=circle] (2.5,1.7) circle (.1) node[above]{$w_1$};
\draw  [shading=ball, ball color=black] (1,1.2) circle (.1) node[left]{$S_1$} node[above]{$a$};

\draw  (3.7,-3.2) node[below]{$\mathcal{H}$} ;
\draw  (1.5,1.2) node[right]{$E$};
\draw  (1.5,1.3) node[above]{$b$} ;
\draw  (2.5,-1.2) node[right]{$F$} ;
\draw [line width=1.2pt] (1,-1.2)--(0,0);
\draw [line width=1.2pt] (1,-1.2)--(2,-2);
\draw [line width=1.2pt] (2,-2)--(3.7,-2.5); 
\draw  (3,-2.3) node[below]{$c$};
\draw  (4.2,-1.9) node[below]{$d$};
\draw [line width=1.2pt] (3.7,-2.5)--(4.7,-1.5) node[left]{$w_2$};
\draw [line width=1.2pt ] (4.7,-1.5)--(5,0) ;
\draw [line width=1.2pt] (5,0)--(4,1.2);
\draw [line width=1.2pt] (4,1.2)--(2.5,1.7);
\draw [line width=1.2pt] (2.5,1.7)--(1,1.2);

\path [pattern=north west lines, pattern  color=blue]   (2,-2)--(4.7,-1.5)--(2.5,1.7)--cycle;

\shade [shading=ball, ball color=black]  (7,0) circle (.1) ;
\draw  [shape=circle] (8,-1.2) circle (.1) node[left]{$S_2$};
\draw  [shape=circle] (9,-2) circle (.1)  node[left]{$w_2$};
\draw  [shape=circle] (10.7,-2.5) circle (.1);
\draw  [shape=circle] (11.7,-1.5) circle (.1);
\draw  [shape=circle] (12,0) circle (.1) ;
\draw  [shape=circle] (11,1.2) circle (.1) ;
\draw  [shape=circle] (9.5,1.7) circle (.1) node[above]{$w_1$};
\draw  [shading=ball, ball color=black] (8,1.2) circle (.1) node[left]{$S_1$} node[above]{$ab$};
\draw  (10,-2.3) node[below]{$c$};
\draw  (11.2,-1.9) node[below]{$d$};
\draw  (10.7,-3.2) node[below]{$\mathcal{H}_1=\mathcal{H}:E$} ;
\draw  (8.5,1.2) node[right]{$E$} ;
\draw  (9.5,-1.2) node[right]{$F$} ;
 \draw [line width=1.2pt] (8,-1.2)--(7,0);
\draw [line width=1.2pt] (8,-1.2)--(9,-2);
\draw [line width=1.2pt] (9,-2)--(10.7,-2.5);
\draw [line width=1.2pt] (10.7,-2.5)--(11.7,-1.5);
\draw [line width=1.2pt ] (11.7,-1.5)--(12,0);
\draw [line width=1.2pt] (12,0)--(11,1.2);
\draw [line width=1.2pt] (11,1.2)--(9.5,1.7);
\path [pattern=north west lines, pattern  color=blue]   (9,-2)--(11.7,-1.5)--(9.5,1.7)--cycle;

\shade [shading=ball, ball color=black]  (14,0) circle (.1) ;
\draw  [shading=ball, ball color=black] (15,-1.2) circle (.1) node[left]{$S_2$};
\draw  [shading=ball, ball color=black] (17.7,-2.5) circle (.1);
\draw  [shape=circle] (18.7,-1.5) circle (.1) ;
\draw  [shape=circle] (19,0) circle (.1) ;
\draw  [shape=circle] (18,1.2) circle (.1) ;
\draw  [shape=circle] (16.5,1.7) circle (.1) node[above]{$w_1$};
\draw  [shading=ball, ball color=black] (15,1.2) circle (.1) node[left]{$S_1$}; 
\draw  (17.4,-2.3) node[below]{$c$};
\draw  (18.2,-1.9) node[below]{$d$};
\draw  (17.7,-3.2) node[below]{$\mathcal{H}_2=\mathcal{H}_{w_2}$} ;
\draw  (17,-.2) node[right]{$F$} ;v
\draw [line width=1.2pt] (15,-1.2)--(14,0);
\draw [line width=1.2pt] (17.7,-2.5)--(18.7,-1.5);
\draw [line width=1.2pt ] (18.7,-1.5)--(19,0);
\draw [line width=1.2pt] (19,0)--(18,1.2);
\draw [line width=1.2pt] (18,1.2)--(16.5,1.7);
\draw [line width=1.2pt] (16.5,1.7)--(15,1.2);
\draw [line width=1.2pt] (18.7,-1.5)--(16.5,1.7);

\shade [shading=ball, ball color=black]  (0,-6) circle (.1) ;
\draw  [shading=ball, ball color=black] (1,-7.2) circle (.1) ;
\draw  [shading=ball, ball color=black] (3.7,-8.5) circle (.1) ;
\draw  [shading=ball, ball color=black] (5,-6) circle (.1) ;
\draw  [shading=ball, ball color=black] (4,-4.8) circle (.1) ;
\draw  [shading=ball, ball color=black] (1,-4.8) circle (.1) node[above]{$ab$};
\draw  [shading=ball, ball color=black] (3,-6.8) circle (.1) node[above]{$x_F$};
\draw  (3.7,-9.2) node[below]{$\mathcal{H}_3=(\mathcal{H},x_F)$} ;
\draw  (3.3,-8.3) node[below]{$cd$};
\draw [line width=1.2pt] (1,-7.2)--(0,-6);
\draw [line width=1.2pt] (5,-6)--(4,-4.8);


\end{tikzpicture}
\end{center}
\end{figure}

\end{example}

\subsection{Splittings - Key Tool}
In \cite{EK} the notion of a \emph{splitting} of a monomial ideal $I$ was
introduced.  

\begin{definition}\cite{EK}\label{splitting}
A monomial ideal $I$ is \emph{splittable} if $I$ is the sum of two
nonzero monomial ideals $J$ and $K$, i.e. $I = J + K$, such that
\begin{enumerate}
\item The generating set $\mathcal{G}(I)$ of $I$,  is the disjoint
  union of $\mathcal{G}(J)$ and $\mathcal{G}(K)$.
\item There is a splitting function 
\[
\begin{aligned}
\mathcal{G}(J\cap K)  &\rightarrow &\mathcal{G}(J) \times
\mathcal{G}(K)\\
w & \rightarrow & (\psi(w), \phi(w))
\end{aligned} 
\]

satisfying
\begin{enumerate}
\item(S1) for all $w \in \mathcal{G}(J \cap K)$, $w = \lcm(\psi(w),
  \phi(w))$.
\item(S2) for every subset $S \subseteq \mathcal{G}(J \cap K)$, both
  $\lcm(\psi(S))$ and $\lcm(\phi(S))$ strictly divide $\lcm(S)$.
\end{enumerate}

\end{enumerate}
If $J$ and $K$ satisfy the above properties they are  called a
\emph{splitting} of $I$.
\end{definition}

Now the key reason we are interested in splittings is the following
result by both Eliahou-Kervaire and separately Fatabbi.

\begin{theorem} (Eliahou-Kervaire \cite{EK} Fatabbi \cite{fatabbi})\label{splitForm}
Suppose $I$ is a splittable monomial ideal with splitting $I = J +
K$.  Then for all $i,j \geq 0$
$$\beta_{i,j} (I) = \beta_{i,j} (J) + \beta_{i,j}(K) +
\beta_{i-1,j}(J\cap K). $$

\end{theorem}

It is important to note that not all monomial ideals admit
splittings. What is interesting is that there are sometimes
monomial ideals that can be decomposed into a sum of ideals $J$ and
$K$ which satisfy the conclusions of the previous theorem.  This
motivates the following definition by Francisco, Ha, and Van Tuyl in \cite{FranHaTu}

\begin{definition}
Let $I, J$ and $K$ be monomial ideals such that $\mathcal{G}(I)$ is the
  disjoint union of $\mathcal{G}(J)$ and $\mathcal{G}(K)$.  Then $I =
  J+K$ is a \emph{Betti splitting} if 
$$\beta_{i,j} (I) = \beta_{i,j} (J) + \beta_{i,j}(K) +
\beta_{i-1,j}(J\cap K) $$
for all $i \in \mathbb{N}$ and all (multi)degrees $j$.
\end{definition}

One complication however is that if one wants to use the existence of
a Betti splitting to prove something about a resolution, one must
first know something about the resolution in question.  The key for us
will be in dissecting the proof of Fatabbi in order to prove that in some special
cases, which may fail condition (S2) in Definition \ref{splitting}, that
a similar formula for (some) Betti numbers holds.

The following lemma is an adaptation of the proof of Fatabbi in a
special case where we do not have a splitting.  In this case we can show that the
necessary conditions hold at the end of the resolution, so that we get
a formula like that of Theorem \ref{splitForm} for the last Betti
numbers.  In particular this allows us to prove statements about projective
dimension.

\begin{lemma}
\label{Split}Let $I$ be a monomial ideal and $I=J+K$ in the ring
$R=\mathbb{K}[x_{1},...,x_{n}]$ over a field $\mathbb{K}$. Suppose
we have the following conditions on projective dimension 
\begin{enumerate}
\item $\pd(R/J)<q$
\item $\pd(R/K)=q$,
\end{enumerate}
and  $\reg(R/K)<r$. 
If $\beta_{q,q+r} (R/J\cap K)\neq 0$ and $\pd (R/J\cap K)=q$, then $\pd(R/I)=q+1$.
\end{lemma}

\begin{proof}
We consider the short exact sequence 
\[
0\rightarrow J\cap K\rightarrow J\oplus K\rightarrow I\rightarrow0.
\]
 Let $\alpha(w)=(w,w)$ be the map from $J\cap K$ to $J\oplus K$
and $\pi(u,v)=u-v$ be the map from $J\oplus K$ to $I$. There is
an induced homology sequence 
\begin{align*}
\cdots & \rightarrow \tor_{q+1}^{R}(J\cap K,\mathbb{K})\rightarrow \tor_{q+1}^{R}(J,\mathbb{K})\oplus \tor_{q+1}^{R}(K,\mathbb{K})\\
 & \rightarrow \tor_{q+1}^{R}(I,\mathbb{K})\rightarrow \tor_{q}^{R}(J\cap K,\mathbb{K})\\
 & \rightarrow \tor_{q}^{R}(J,\mathbb{K})\oplus \tor_{q}^{R}(K,\mathbb{K})\rightarrow \cdots
\end{align*}
We have $\pd(R/J)<q$, $\pd(R/K)=q$,
and $\pd(R/(J\cap K))=q$, hence the short exact sequence gives $\pd(R/I)\leq \max\{\pd(R/(J\cap K))+1,\pd(R/J\oplus K)\}\leq q+1$.  The homology sequence becomes 
\[
0\rightarrow0\rightarrow \tor_{q+1}^{R}(I,\mathbb{K})\rightarrow
\tor_{q}^{R}(J\cap K,\mathbb{K})\rightarrow \tor_{q}^{R}(K,\mathbb{K}) \rightarrow \cdots
\]

Moreover, the $q+r$ graded piece is the following:
\[
0\rightarrow \tor_{q+1}^{R}(I,\mathbb{K})_{q+r}\rightarrow
\tor_{q}^{R}(J\cap K,\mathbb{K})_{q+r}\rightarrow
\tor_{q}^{R}(K,\mathbb{K})_{q+r} \rightarrow \cdots
\]
 
Now using our assumption that $\reg(R/K)<r$ and $\pd(R/K)=q$
then we have $\tor_{q}^{R}(K,\mathbb{K})_{q+r}=0$.
This shows $\tor_{q+1}^{R}(I,\mathbb{K})_{q+r}\cong \tor_{q}^{R}(J\cap K,\mathbb{K})_{q+r}\neq0$
by the fact that $\beta_{q,q+r} (R/J\cap K)\neq 0$
and hence $\pd(R/I)=q+1$.
\end{proof}

\begin{remark}\label{SplitH}
The above lemma can be translated in terms of associated hypergraphs. Let
$\mathcal{H} = \mathcal{H}(I)$ be a hypergraph with underlying vertex set $V$, and let $V_1$ and $V_2$ be a partition
of the vertices of $\mathcal{H}$ such that $V_1 \cup V_2 = V$ and
$V_1 \cap V_2$ is empty.  Now define $I_i$ to be the ideal generated
by the generators of $I$ indexed by the elements in $V_i$ for $i =
1,2$.  Let $G_i=\mathcal{H}(I_i)$ for $i=1,2$ and $\mathcal{H}(I_1\cap
I_2)$ be the hypergraphs corresponding the ideals $I_i$ for $i=1,2$
and the ideal $I_1\cap I_2$.  Suppose $\pd(G_1)<q$, $\pd(G_2)=q$,  $\pd(\mathcal{H}(I_1\cap I_2))=q$,
and $\reg(G_2)<r$ and $\reg(\mathcal{H}(I_1\cap I_2))=r$. Suppose $\beta_{q,q+r}(\mathcal{H}(I_1\cap I_2))\neq 0$, then $\pd(\mathcal{H})=q+1$. 
\end{remark}

\section{Strings with higher dimensional edges}  \label{StringsSec}

In this section we are primarily interested in finding the projective dimension of a square-free monomial ideal such that its hypergraph is a string  with higher dimensional
edges attached to it.  Our primary object is described below and we fix the notation now for the easy reference later. 

\begin{notation}\label{stringwithEdge}\hangindent\leftmargini
\hskip\labelsep
Let $\mathcal{H}$ be a hypergraph and $F$ be an edge of $\mathcal{H}$. 
\begin{enumerate}[label={(\arabic*)},ref={(\arabic*)}]
\item Let $\mathcal{H}_{S_\mu}$ be a string hypergraph with vertex set $V=\{w_1, \dots,
w_\mu\}$, and an edge $F$ with $k>1$ vertices $\{v_1=w_{i_1}, \dots, v_k=w_{i_k}
\}\subseteq \{w_1, \dots,
w_\mu\}$ with $1\leq i_1< \dots <i_k \leq \mu$. When we say $\mathcal{H}$ is \emph{ a string together with an edge consisting of $k$ vertices}, we mean $\mathcal{H}=\mathcal{H}_{S_\mu} \cup F$. 
\item Let $n_1=i_1-1$, $n_j=i_j-i_{j-1}-1$ for $j=2, \dots k$, and $n_{k+1}=\mu-i_k$ be the number of vertices between vertices of $F$.
\item We write $n_i=3l_i+r_i$  where $l_{i}$ are some
non-negative integers, and $0\leq r_{i}\leq2$ for $i=1,...,k+1$.

\end{enumerate}
\end{notation}

\begin{example} \label{2k-1}
The hypergraph shown in Figure \ref{StringedgeFig} shows the string
hypergraph with a higher dimensional edge with the notation outlined
above.  In this case, $\mu = 11$, $k = 4$, $n_1 = 1$, $n_2 = n_3 = n_4= 2$, and $n_5 = 0$.

\begin{figure}[h] 
\caption{}\label{StringedgeFig}
\begin{center}
\begin{tikzpicture}[thick, scale=0.8]
\shade [shading=ball, ball color=black]  (0,0) circle (.1) node[left]{$w_1$};
\draw  [shape=circle] (1,-0.5) circle (.1) node[left]{$w_2 = v_1$} ;
\draw  [shape=circle] (2,-1) circle (.1)  node[below]{$w_3$};
\draw  [shape=circle] (3,-1.5) circle (.1) node[below]{$w_4$} ;
\draw  [shape=circle] (4,-1) circle (.1) node[right]{$w_5 = v_2$} ;
\draw  [shape=circle] (5,-0.5) circle (.1) node[right]{$w_6$} ;
\draw  [shape=circle] (6,0) circle (.1) node[right]{$w_7$} ;
\draw  [shape=circle] (5,0.5) circle (.1) node[right]{$w_8= v_3$} ;
\draw  [shape=circle] (4,1) circle (.1) node[above]{$w_9$} ;
\draw  [shape=circle] (3,1.5) circle (.1) node[above]{$w_{10}$} ;
\shade [shading=ball, ball color=black]  (2,1) circle (.1)
node[left]{$w_{11} = v_4$};

\draw [line width=1.2pt] (0,0)--(1,-0.5);
\draw [line width=1.2pt] (1,-0.5)--(2,-1);
\draw [line width=1.2pt] (2,-1)--(3,-1.5);
\draw [line width=1.2pt] (3,-1.5)--(4,-1);
\draw [line width=1.2pt ] (4,-1)--(5,-0.5);
\draw [line width=1.2pt] (5,-0.5)--(6,0);
\draw [line width=1.2pt] (6,0)--(5,0.5);
\draw [line width=1.2pt] (5,0.5)--(4,1);
\draw [line width=1.2pt] (4,1)--(3,1.5);
\draw [line width=1.2pt] (3,1.5)--(2,1);
\path [pattern=north west lines, pattern  color=blue]   (1,-0.5)--(4,-1)--(5,0.5)--(2,1)--cycle;

\end{tikzpicture}
\end{center}
\end{figure}
\end{example}

Because of
\Cref{LaundryList1} \ref{removeMeets}, we will primarily focus on the higher dimensional
edges which are not unions of two or more edges. The vertex $w_i$ is
assumed to be open for all $i$ unless otherwise stated.  We
find the projective dimensions by considering three different cases:
$\sum_{i=1}^{k+1}r_{i}<2k$, $\sum_{i=1}^{k+1}r_{i}\geq2k+1$, and
$\sum_{i=1}^{k+1}r_{i}=2k$ in three propositions. We conclude this
section with \Cref{StringedgeAll} which covers all the previous
results. \Cref{2k-1} is an example of the cases when $\sum_{i=1}^{k+1}r_{i}<2k$.

First we will prove two propositions that deal with the case when $\sum_{i=1}^{k+1}r_{i}<2k$,
 or $\sum_{i=1}^{k+1}r_{i}\geq2k+1$. In these two cases we will show that for a hypergraph $\mathcal{H}$
satisfying the hypotheses of the propositions, that the projective dimension
will be the same as for $\mathcal{H}_{S_\mu}$. Before we can get to these propositions though we present couple of technical computations first, and the computational processes will be used in a similar manner in the later proofs. We often use the following identities without justification $A=\left\lfloor \frac{3A}{3}\right\rfloor =\left\lfloor \frac{3A+1}{3}\right\rfloor =\left\lfloor \frac{3A+2}{3}\right\rfloor$  and $A-1=\left\lfloor \frac{3A-3}{3}\right\rfloor =\left\lfloor \frac{3A-1}{3}\right\rfloor =\left\lfloor \frac{3A-2}{3}\right\rfloor$. 

\begin{lemma}
\label{lem:<2k}Let $n_{i}=3l_{i}+r_{i}$ with $0\leq r_{i}<3$ for
$i=1,...,k+1$. If $\sum_{i=1}^{k+1}r_{i}<2k$, then \[\sum_{i=1}^{k+1}(n_{i}-\left\lfloor \frac{n_{i}}{3}\right\rfloor )<k+\sum_{i=1}^{k+1}n_{i}-\left\lfloor \frac{k+\sum_{i=1}^{k+1}n_{i}}{3}\right\rfloor .\]

\end{lemma}

\begin{proof}
We simplify the left hand side of the inequality using $n_{i}=3l_{i}+r_{i}$
and $0\leq r_{i}<3$. 
\begin{align*}
\sum_{i=1}^{k+1}(n_{i}-\left\lfloor \frac{n_{i}}{3}\right\rfloor ) & =\sum_{i=1}^{k+1}(n_i-\left\lfloor \frac{3l_{i}+r_{i}}{3}\right\rfloor )\\
 & =\sum_{i=1}^{k+1}(n_i-l_{i})
\end{align*}
We simplify the right hand side of the inequality with the assumption of
$\sum_{i=1}^{k+1}r_{i}<2k$ to obtain the conclusion as follows.
\begin{align*}
k+\sum_{i=1}^{k+1}n_{i}-\left\lfloor \frac{k+\sum_{i=1}^{k+1}n_{i}}{3}\right\rfloor  & =k+\sum_{i=1}^{k+1}n_i-\left\lfloor \frac{k+\sum_{i=1}^{k+1}(3l_{i}+r_{i})}{3}\right\rfloor \\
 & >k+\sum_{i=1}^{k+1}n_i-\left\lfloor \frac{k+2k+\sum_{i=1}^{k+1}(3l_{i})}{3}\right\rfloor \\
 & =k+\sum_{i=1}^{k+1}n_i-k-\sum_{i=1}^{k+1}l_{i}\\
 & =\sum_{i=1}^{k+1}(n_i-l_{i})
\end{align*}
\end{proof}
\begin{lemma}
\label{lem:>2k+1}Let $n_{i}=3l_{i}+r_{i}$ with $r_{1}=2$ and $0\leq r_{i}<3$
for $i=1,...,k+1$. If $\sum_{i=1}^{k+1}r_{i}\geq2k+1$, then \begin{enumerate}
    \item $k+\sum_{i=1}^{k+1}n_{i}-\left\lfloor \frac{k+\sum_{i=1}^{k+1}n_{i}}{3}\right\rfloor =\sum_{i=1}^{k+1}(n_{i}-\left\lfloor \frac{n_{i}}{3}\right\rfloor )$
    \item $n_{1}-\left\lfloor \frac{n_{1}-2}{3}\right\rfloor +k-2+\sum_{i=2}^{k+1}n_{i}-\left\lfloor \frac{k+\sum_{i=2}^{k+1}n_{i}}{3}\right\rfloor <\sum_{i=1}^{k+1}(n_{i}-\left\lfloor \frac{n_{i}}{3}\right\rfloor )$.
\end{enumerate}

\end{lemma}

\begin{proof}
Notice that $2k+2\geq\sum_{i=1}^{k+1}r_{i}\geq2k+1$ by the assumption $0\leq r_{i}<3$.  Where the first inequality assumes that each $r_i = 2$. This means that $\sum_{i=1}^{k+1}r_{i}$ is either $2k+1$ or $2k+2$, hence we must have
\begin{align*}
\left\lfloor \frac{k+\sum_{i=1}^{k+1}n_{i}}{3}\right\rfloor 	&=\left\lfloor \frac{k+\sum_{i=1}^{k+1}(3l_{i}+r_{i})}{3}\right\rfloor \\
&=\left\lfloor \frac{k+\sum_{i=1}^{k+1}(3l_{i})+\sum_{i=1}^{k+1}r_{i}}{3}\right\rfloor \\
&=k+\sum_{i=1}^{k+1}l_{i}\\
&=k+\sum_{i=1}^{k+1}\left\lfloor \frac{n_{i}}{3}\right\rfloor .
\end{align*}
The first equality follows immediately.

The inequality (2) requires the extra assumption $r_{1}=2$, since together with the assumption that $\sum_{i=1}^{k+1}r_{i}\geq2k+1$
it implies $\sum_{i=2}^{k+1}r_{i}\geq2k-1$. We write $n_{1}=3l_{1}+2$
and $n_{i}=3l_{i}+r_{i}$, and use the inequality $\sum_{i=2}^{k+1}r_{i}\geq2k-1$ to obtain the statement's inequality.
\begin{align*}
&n_{1}-\left\lfloor \frac{n_{1}-2}{3}\right\rfloor +k-2+\sum_{i=2}^{k+1}n_{i}-\left\lfloor \frac{k+\sum_{i=2}^{k+1}n_{i}}{3}\right\rfloor  \\ &=n_{1}-\left\lfloor \frac{3l_{1}}{3}\right\rfloor+k-2+\sum_{i=2}^{k+1}n_{i}-\left\lfloor \frac{k+\sum_{i=2}^{k+1}(3l_{i}+r_{i})}{3}\right\rfloor \\
 & \leq n_{1}-l_{1}+k-2+\sum_{i=2}^{k+1}n_{i}-\left\lfloor \frac{k+2k-1+\sum_{i=2}^{k+1}(3l_{i})}{3}\right\rfloor \\
 & =k-2+\sum_{i=1}^{k+1}n_{i}-k+1-\sum_{i=1}^{k+1}l_{i}\\
 & <\sum_{i=1}^{k+1}(n_{i}-\left\lfloor \frac{n_{i}}{3}\right\rfloor ).
\end{align*}

\end{proof}

With these technical computations out of the way we are ready to state and prove the first two propositions.  

\begin{proposition}
\label{String2k-1} We adapt  \Cref{stringwithEdge}. If $\sum_{i=1}^{k+1}r_{i}<2k$,
then $\pd(\mathcal{H})=\pd(\mathcal{H}_{S_{\mu}})$.
\end{proposition}

\begin{proof}
Notice that $\mu=k+\sum_{i=1}^{k+1}n_{i}$ and $\pd(\mathcal{H}_{S_{\mu}})=\mu-\left\lfloor \frac{\mu}{3}\right\rfloor =k+\sum_{i=1}^{k+1}n_{i}-\left\lfloor \frac{k+\sum_{i=1}^{k+1}n_{i}}{3}\right\rfloor $.
Now consider the short exact sequence
\[
0\rightarrow (\mathcal{H}:F) \rightarrow \mathcal{H}\rightarrow(\mathcal{H},x_F)\rightarrow0
\]
 where $x_F$ is the variable corresponding to $F$. We first observe
that $(\mathcal{H}:F)=\mathcal{H}_{S_{\mu}}$ hence $\pd(\mathcal{H}_{S_{\mu}})\leq\pd(\mathcal{H})$ by Theorem \ref{LaundryList1} \ref{LMa2-2.9c}.
Let $\mathcal{H}_{V_{F}}=\mathcal{H}_{S_{\mu}}\cap(V\backslash V_{F})$ be the hypergraph
obtained by removing all the vertices $v_{1},...,v_{k}$. Then $\mathcal{H}_{V_{F}}\cup\{x_F\}=(\mathcal{H},x_F)$.
Notice that $\mathcal{H}_{V_{F}}$ is union of $k+1$-strings such that the $i$-th string
has $n_{i}$ vertices so by  \Cref{LaundryList1} \ref{pdFormulaString}
we get that
$\pd(\mathcal{H}_{V_{F}})=\sum_{i=1}^{k+1}(n_{i}-\left\lfloor
  \frac{n_{i}}{3}\right\rfloor )$. Once we show that $\pd(\mathcal{H}_{V_{F}})<\pd(\mathcal{H}_{S_{\mu}})$,
then by the short exact sequence, we have \[\pd(\mathcal{H}_{S_{\mu}})\leq\pd(\mathcal{H})\leq\max\{\pd(\mathcal{H}_{S_{\mu}}),\pd(\mathcal{H}_{V_{F}})+1\}=\pd(\mathcal{H}_{S_{\mu}}).\]
To show that $\pd(\mathcal{H}_{V_{F}})<\pd(\mathcal{H}_{S_{\mu}})$ it is sufficient to show that 
\[
\sum_{i=1}^{k+1}(n_{i}-\left\lfloor \frac{n_{i}}{3}\right\rfloor )< k+\sum_{i=1}^{k+1}n_{i}-\left\lfloor \frac{k+\sum_{i=1}^{k+1}n_{i}}{3}\right\rfloor 
\]
 which is shown in \Cref{lem:<2k}.
 \end{proof}

\begin{proposition}
\label{String2k+1} We adapt  \Cref{stringwithEdge}.  If $\sum_{i=1}^{k+1}r_{i}\geq2k+1$,
then $\pd(\mathcal{H})=\pd(\mathcal{H}_{S_{\mu}})$.
\end{proposition}
\begin{proof}
 We first notice that $\pd(\mathcal{H}_{S_{\mu}})=k+\sum_{i=1}^{k+1}n_{i}-\left\lfloor \frac{k+\sum_{i=1}^{k+1}n_{i}}{3}\right\rfloor =\sum_{i=1}^{k+1}(n_{i}-\left\lfloor \frac{n_{i}}{3}\right\rfloor )$ by \Cref{LaundryList1} \ref{pdFormulaString} and \Cref{lem:>2k+1} (1).  Since $\sum_{i=1}^{k+1}r_{i}\geq2k+1$, $r_{i}<3$, and $k>1$, we have at most one $r_{i}$ is equal to $1$. We may assume $r_{1}=2$. Let $\mathcal{H}_{v_{1}}=\mathcal{H}_{v}$
be the hypergraph where we remove the vertex $v_{1}=v$ from $\mathcal{H}$
and let $\mathcal{H}_{v}:v_{1}=\mathcal{Q}_{v}$ be the hypergraph $\mathcal{H}(I_{v}:m_{v})$ where
$I_{v}=I(\mathcal{H}_{v})$ and $m_{v}$ is the monomial corresponding to the
vertex $v$. We have a short exact sequence 
\[
0\rightarrow \mathcal{Q}_{v}\rightarrow \mathcal{H}_{v}\rightarrow \mathcal{H}\rightarrow0.
\]
 We claim that in this case $\pd(\mathcal{H}_{S_{\mu}})=\pd(\mathcal{H}_{v})>\pd(\mathcal{Q}_{v})$, then by \Cref{shortexact},
we conclude that $\pd(\mathcal{H}_{S_{\mu}})=\pd(\mathcal{H}_{v})=\pd(\mathcal{H})$. 

To see the proof of the claim, we use induction on $k$. When $k=2$, $\mathcal{H}_{v}$ is a union of two strings of length $n_{2}+n_{3}+1$ and  $n_{1}$. When $n_2\geq2$ and $n_3\geq2$, the string of length $n_{2}+n_{3}+1$ has two open strings with $n_{2}-1$ and $n_{3}-1$ open vertices. When ($n_2=1$ and $r_3=2$), or ($n_3=1$ and $r_2=2$), the string of length $n_{2}+n_{3}+1$ has exactly $3$ closed vertices at the ends of string and all other vertices are open.
By the work of \cite{LMa1}, \Cref{LaundryList1} \ref{pdFormulaString},
we have either
\begin{align*}
\pd(\mathcal{H}_{v}) & =n_{1}-\left\lfloor \frac{n_{1}}{3}\right\rfloor +n_{2}+n_{3}+1-2-\left\lfloor \frac{n_{2}-2}{3}\right\rfloor -\left\lfloor \frac{n_{3}-2}{3}\right\rfloor +1\\
 & =\sum_{i=1}^{3}(n_{i}-\left\lfloor \frac{n_{i}}{3}\right\rfloor ) =\pd(\mathcal{H}_{S_{\mu}})
\end{align*}
when $n_2\geq2$, $n_3\geq2$, $r_2=2$ and $r_3=2$, or
\begin{align*}
\pd(\mathcal{H}_{v}) & =n_{1}-\left\lfloor \frac{n_{1}}{3}\right\rfloor +n_{2}+n_{3}+1-2-\left\lfloor \frac{n_{2}-2}{3}\right\rfloor -\left\lfloor \frac{n_{3}-2}{3}\right\rfloor \\
 & =\sum_{i=1}^{3}(n_{i}-\left\lfloor \frac{n_{i}}{3}\right\rfloor )=\pd(\mathcal{H}_{S_{\mu}})
\end{align*}
when $n_2\geq2$, $n_3\geq2$, ($r_2=1$ and $r_3=2$) or ($r_2=2$ and $r_3=1$), or
 \begin{align*}
\pd(\mathcal{H}_{v}) & =n_{1}-\left\lfloor \frac{n_{1}}{3}\right\rfloor +n_{2}+n_{3}+1-1 -\left\lfloor \frac{n_{3}-2}{3}\right\rfloor \\
 & =\sum_{i=1}^{3}(n_{i}-\left\lfloor \frac{n_{i}}{3}\right\rfloor )=\pd(\mathcal{H}_{S_{\mu}})
\end{align*}
when $n_2=1$ and $r_3=2$, or 
\begin{align*}
\pd (\mathcal{H}_{v}) & =n_{1}-\left\lfloor \frac{n_{1}}{3}\right\rfloor +n_{2}+n_{3}+1-1-\left\lfloor \frac{n_{2}-2}{3}\right\rfloor \\
 & =\sum_{i=1}^{3}(n_{i}-\left\lfloor \frac{n_{i}}{3}\right\rfloor )=\pd(\mathcal{H}_{S_{\mu}})
\end{align*}
when $n_3=1$ and $r_2=2$. 
On the other hand, $\mathcal{Q}_{v}$ is a union of two isolated vertices
and two strings of length, $n_{1}-2$, $n_{2}-2+n_{3}+1$, hence we
have 
\begin{align*}
\pd(\mathcal{Q}_{v}) & =2+n_{1}-2-\left\lfloor \frac{n_{1}-2}{3}\right\rfloor +n_{2}+n_{3}-1-\left\lfloor \frac{n_{2}+n_{3}-1}{3}\right\rfloor \\
 & \leq \sum_{i=1}^{3}(n_{i}-\left\lfloor \frac{n_{i}}{3}\right\rfloor )-1 <\pd(\mathcal{H}_{v}).
\end{align*}
For the second inequality above, we use the fact that $r_2+r_3\geq
3$. 

For the case when $k>2$, we use the same exact sequence. Here the hypergraph
$\mathcal{H}_{v}$ is a union of a string of $n_{1}$ vertices and
a string of length $\mu'=k-1+\sum_{i=2}^{k+1}n_{i}$ with a $k-2$-dimensional edge of $k-1$ vertices such that $\sum_{i=2}^{k+1}r_{i}\geq2(k-1)+1$.
By the induction hypothesis and \Cref{LaundryList1} \ref{pdFormulaString},
$\pd(\mathcal{H}_{v})=n_{1}-\left\lfloor \frac{n_{1}}{3}\right\rfloor +\pd(\mathcal{H}_{S_{\mu'}})=\sum_{i=1}^{k+1}(n_{i}-\left\lfloor \frac{n_{i}}{3}\right\rfloor) =\pd(\mathcal{H}_{S_\mu})$.
On the other hand, $\mathcal{Q}_{v}$ is a union of two isolated closed vertices
and two strings of length $n_{1}-2$ and $k-3+\sum_{i=2}^{k+1}n_{i}$.
Hence we have 
\begin{align*}
\pd(\mathcal{Q}_{v}) & =2+n_{1}-2-\left\lfloor \frac{n_{1}-2}{3}\right\rfloor +k-3+\sum_{i=2}^{k+1}n_{i}-\left\lfloor \frac{k-3+\sum_{i=2}^{k+1}n_{i}}{3}\right\rfloor \\
& =n_{1}-\left\lfloor \frac{n_{1}-2}{3}\right\rfloor +k-2+\sum_{i=2}^{k+1}n_{i}-\left\lfloor \frac{k+\sum_{i=2}^{k+1}n_{i}}{3}\right\rfloor \\
 & <\sum_{i=1}^{k+1}(n_{i}-\left\lfloor \frac{n_{i}}{3}\right\rfloor)=\pd\mathcal{H}_{v}
\end{align*}
where the inequality is shown in \Cref{lem:>2k+1} (2).  We conclude $\pd(\mathcal{H})=\pd(\mathcal{H}_{v})=\pd(\mathcal{H}_{S_{\mu}})$.
\end{proof}

Now we want to deal with the case when $\sum_{i=1}^{k+1}r_{i}=2k$.  In this case we get two different outcomes, and it
will be necessary to prove a number of lemmas that will allow us to
work with the special case.  The primary strategy is using induction on the number of vertices of the extra edge $F$, and the short exact sequences with \Cref{LMa247} to remove vertices or edges to reduce cases into smaller cases. Lemmas are ordered in a way that later lemmas are built or proven with the previous lemmas' conclusions.

First we will need to to prove some results about when the spacing
measured by the $n_i$ is equivalent to 2 modulo 3.  The following lemma deals with the case where one of the ends of the
string coincides with a vertex from $F$.

\begin{lemma}
\label{Stringedge}We adapt  \Cref{stringwithEdge}. Suppose the end vertices of the string are $w_{1}$ and $w_{\mu}$ where $w_{1}$
is closed and $w_{\mu}$ is open, and $w_{\mu}=v_{k}$. 
If $n_{i}=2+3l_{i}$ for $i=1,...,k$, then the projective dimension of $\mathcal{H}$ is \[ 2k+2\sum_{i=1}^{k}l_{i}=k+\sum_{i=1}^{k}n_{i}-\left\lfloor \frac{k+\sum_{i=1}^{k}n_{i}}{3}\right\rfloor \]
and \[\reg(\mathcal{H})\leq k+\sum_{i=1}^{k}l_{i}.\]
\end{lemma}

\begin{proof}
We use induction on $k$. When $k=1$, we have $w_{\mu}=v_{1}$. In this
case the edge has only one vertex $v_{1}$ which forces $w_{\mu}=v_{1}$
to become a closed
vertex.  Also, $\mathcal{H}$ becomes a string of length $\mu=n_{1}+1$,
so by \Cref{LaundryList1} \ref{pdFormulaString}, $\pd(\mathcal{H})=n_{1}+1-\left\lfloor \frac{n_{1}+1}{3}\right\rfloor =2+2l_{1}$
and $\reg(\mathcal{H})=\left\lceil \frac{n_{1}+1}{3}\right\rceil
=1+l_{1}$.

For the induction step, we consider the short exact sequence
\[
0\rightarrow \mathcal{H}_{v_{1}}:v_{1}=\mathcal{Q}_{v_{1}} \rightarrow \mathcal{H}_{v_{1}}\rightarrow  \mathcal{H} \rightarrow0.
\]
We will show $\pd(\mathcal{H}_{v_{1}})=2k+2\sum_{i=1}^{k}l_{i}>\pd(\mathcal{Q}_{v_1})$, then by \Cref{shortexact}, we conclude $\pd(\mathcal{H})=\pd(\mathcal{H}_{v_{1}})=2k+2\sum_{i=1}^{k}l_{i}$. 

Since $k>1$, the vertex $v_{1}$ corresponds to a monomial of
degree $3$. Notice that $\mathcal{H}_{v_{1}}$ is the union of a string of length
$n_{1}$ and a hypergraph with exactly the same structure of $\mathcal{H}$
(i.e. closed vertex on one end of string and an open vertex coinciding
with a vertex of the higher dimensional edge at the other end)
such that it has an edge with $k-1$ vertices. By \Cref{LaundryList1} \ref{pdFormulaString}
and induction hypothesis, we have \[\pd(\mathcal{H}_{v_{1}})=n_{1}-\left\lfloor \frac{n_{1}}{3}\right\rfloor +2(k-1)+2\sum_{i=2}^{k}l_{i}=2k+2\sum_{i=1}^{k}l_{i}\]
and \[\reg(\mathcal{H}_{v_{1}})\leq\left\lceil \frac{n_{1}}{3}\right\rceil +k-1+\sum_{i=2}^{k}l_{i}=k+\sum_{i=1}^{k}l_{i}.\]

Moreover, the hypergraph $\mathcal{Q}_{v_{1}}$ is a union of three isolated vertices
and two strings of length $n_{1}-2$ and $n_{2}-2+n_{k}-1+k-2+\sum_{i=3}^{k-1}n_{i}$.
Therefore by \Cref{LaundryList1} \ref{pdFormulaString} with the fact that $\sum_{i=2}^{k}r_{i}=2(k-1)$, we have 
\begin{align*}
\pd(\mathcal{Q}_{v}) &=3+n_{1}-2-\left\lfloor \frac{n_{1}-2}{3}\right\rfloor
  +k-5+\sum_{i=2}^{k}n_{i}-\left\lfloor
  \frac{k-5+\sum_{i=2}^{k}n_{i}}{3}\right\rfloor \\
  &=2l_{1}
  +k-2+\sum_{i=2}^{k}(3l_{i}+2)-\left\lfloor
  \frac{k-5+\sum_{i=2}^{k}(3l_{i}+2)}{3}\right\rfloor\\
  &=2l_{1}
  +k-2+2(k-1)+\sum_{i=2}^{k}(3l_{i})-\left\lfloor
  \frac{k-5+2(k-1)+\sum_{i=2}^{k}(3l_{i})}{3}\right\rfloor\\
  &=3k-4+2\sum_{i=1}^{k}l_{i}-k+3=2k+2\sum_{i=1}^{k}l_{i}-1
\end{align*}

and \[\reg(\mathcal{Q}_{v})=\left\lceil \frac{n_{1}-2}{3}\right\rceil +\left\lceil \frac{k-5+\sum_{i=2}^{k}n_{i}}{3}\right\rceil =k+\sum_{i=1}^{k}l_{i}-2.\]
We have shown \[\pd(\mathcal{H})=\pd(\mathcal{H}_{v_{1}})=2k+2\sum_{i=1}^{k}l_{i}.\] Using the short exact sequence and \Cref{LaundryList1} \ref{RegFormula2}, we have 
and \[\reg(\mathcal{H})\leq\max\{k+\sum_{i=1}^{k}l_{i},k+\sum_{i=1}^{k}l_{i}-2+3-1\}=k+\sum_{i=1}^{k}l_{i}.\] 
\end{proof}

In the next lemma we need to consider a specific case which is necessary for the proof of
Lemma \ref{SubStinky}.  Specifically this will address
$\mathcal{H}_{v_2}$ in Lemma \ref{SubStinky} where $\mathcal{H}$ is a a string together with
  an edge consisting of $k$ vertices. Those two special cases are shown in \Cref{ex:substinky} later. 


\begin{lemma}\label{2Stringedge}

We adapt  \Cref{stringwithEdge}. Assume $w_{1}=v_{1}$ and $w_{\mu}=v_{k}$, i.e. $n_1=n_{k+1}=0$, and both $w_{1}=v_{1}$ and $w_{\mu}=v_{k}$ are open vertices. If $n_{i}=2+3l_{i}$ for $i=2,...,k$.  Then the projective dimension of $\mathcal{H}_{v_2}$ is $2(k-1)+2\sum_{i=2}^{k}l_{i}$ and $\reg(\mathcal{H}_{v_2})\leq k-1+\sum_{i=2}^{k}l_{i}$.

\end{lemma}

\begin{proof}
We use induction on $k$. When  $k=2$, $\mathcal{H}_{v_2}$ is a string of length $n_{2}+1$. By \Cref{LaundryList1} \ref{pdFormulaString}, $\pd(\mathcal{H}_{v_2})=n_{2}+1-\left\lfloor \frac{n_{2}+1}{3}\right\rfloor =3l_2+3-\left\lfloor \frac{3l_{2}+3}{3}\right\rfloor =2+2(l_{2})$ and $\reg(\mathcal{H}_{v_2})= \left\lfloor \frac{n_2+1}{3}\right\rfloor=1+l_2 $. 
For the induction step, we consider the short exact sequence 
\[
0\rightarrow \mathcal{H}_{v_{1},v_{2}}:v_{1} \rightarrow \mathcal{H}_{v_{1},v_2}\rightarrow  \mathcal{H}_{v_2} \rightarrow0,
\]
where $\mathcal{H}_{v_{1},v_2}$ is the hypergraph obtained from $\mathcal{H}$ after removing vertices $v_1$ and $v_2$. The proof is almost identical to the Lemma \ref{Stringedge} except
that $v_{1}$ corresponds to a monomial of degree $2$. $\mathcal{H}_{v_{1},v_2}$
is a union of a string of length $n_2$ and a hypergraph satisfying the assumptions of Lemma
\ref{Stringedge}, and $\mathcal{H}_{v_{1},v_{2}}:v_{1}$ is a union of two isolated vertices, and two strings of length $n_2-2$ and $n_{k}-1$ when $k=3$, and $n_{k}-1+k-3+\sum_{i=3}^{k-1}n_i$ when $k>3$. We have $\reg(\mathcal{H}_{v_{1},v_{2}}:v_{1})= k-1+\sum_{i=2}^{k}l_{i}-1$ and $\pd(\mathcal{H}_{v_{1},v_{2}}:v_{1})= 2(k-1)+2\sum_{i=2}^{k}l_{i}-1$. 
Hence we have $\pd(\mathcal{H}_{v_2})=\pd(\mathcal{H}_{v_{1},v_2})=2(k-1)+2\sum_{i=2}^{k}l_{i}$
and $\reg(\mathcal{H}_{v_2})\leq\max\{k-1+\sum_{i=2}^{k}l_{i},k-1+\sum_{i=2}^{k}l_{i}-1
+2-1\}=k-1+\sum_{i=2}^{k}l_{i}$ by \Cref{LaundryList1} \ref{RegFormula2}.
\end{proof}

Now we will use the splitting type result in Lemma \ref{Split} to
finish our necessary results for the hypergraphs which are a string together with
an edge consisting of $k$ vertices, where the spacing between the vertices of the edge are equivalent to 2 modulo 3.

The next lemma deals with the intersection ideal in the special case
where $G_1$ will correspond to one vertex of a larger hypergraph $\mathcal{H}$.

\begin{lemma}\label{IntersectionSplit}
We adapt  \Cref{stringwithEdge}. Assume $k>2$, $w_{1}=v_{1}$ and $w_{\mu}=v_{k}$, i.e. $n_1=n_{k+1}=0$, and both $w_{1}=v_{1}$ and $w_{\mu}=v_{k}$ are open vertices. Furthermore, $n_{i}=2+3l_{i}$ for $i=2,...,k$. Let $G_1 = \{v_2\}$ and $G_2 =
\mathcal{H}_{v_2}$ (which is just $\mathcal{H}$ with $v_2$ removed)
and denote $I_i=I(G_i)$ as the ideals corresponding to $G_i$,
then $I_1 \cap I_2 = m_{v_2}I'$ where $m_{v_2}$ is the monomial corresponding to the vertex $v_2$ and $\mathcal{H}(I')$ has 4 isolated
vertices and 2 strings one of length:
\begin{itemize}
\item  $n_2 -3$ when $n_2 > 2$, or
\item  $0$ if $n_2 = 2$

\end{itemize}
 and the other of length:
\begin{itemize}
\item $k-3 +n_k -1 +n_3-2 +\sum_{i=4}^{k-1} n_i$ when $k > 3$, or
\item $n_3 -3$ when $k = 3$ and when $n_3 > 2$, or
\item $0$ when $k = 3$ and $n_3 = 2$
\end{itemize}

\end{lemma}
\begin{proof}
To see this
consider the hypergraph $\mathcal{H}$, for notational convenience,
let us denote the vertex neighboring $v_{1}$ as $w_{\alpha}$, the
vertices neighboring $v_{2}$ as $w_{\beta_{1}}$ and $w_{\beta_{2}}$,
and the vertex neighboring $v_{k}$ as $w_{\gamma}$. Now removing
$v_{2}$ from $\mathcal{H}$ leaves us with a hypergraph on the same
vertex set excluding $v_{2}$ and all vertices remain open except
$w_{\beta_{1}}$ and $w_{\beta_{2}}$ which become closed, together
with a $k-1$ edge $F'$ that has $\{v_{1},v_{3},\dots,v_{k}\}$ as
its vertex set (note this also describes the hypergraph $\mathcal{H}_{v_{2}}$).

Now we consider the intersection with the ideal generated by $G_{1}$.
A first step towards finding these new generators is to multiply each
generator for $G_{2}$ by the monomial $m_{v_2}$ corresponding to $v_{2}$
in the original $\mathcal{H}$. The result on hypergraphs is now we
get a hypergraph, which we will denote as $m_{v_2}\mathcal{H}_{v_{2}}$, consisting
of 2 strings of only open vertices where one string is the part of
$\mathcal{H}$ consisting of $v_{1}$ to $w_{\beta_{1}}$ and the
other is $w_{\beta_{2}}$ to $v_{k}$, and all the open vertices are
in an edge corresponding to $m_{v_2}$. Note that the spacing measurements
for $m_{v_2}\mathcal{H}_{v_{2}}$ are the same as for $\mathcal{H}_{v_{2}}$.
The issue here is that $m_{v_2}\mathcal{H}_{v_{2}}$ is not separated (i.e.
the generators are not minimal). Denote the vertices neighboring each
$w_{\beta_{i}}$ as $w_{\beta_{i}'}$. It is easy to see that removing
vertices $w_{\alpha}$, $w_{\beta_{1}'}$, $w_{\beta_{2}'}$, and
$w_{\gamma}$ from $m_{v_2}\mathcal{H}_{v_{2}}$ produces the desired separated
hypergraph corresponding to $I_{1}\cap I_{2}$. 

Notice that when
$n_{2}=2$ then $w_{\alpha}=w_{\beta_{1}'}$ and similarly if $k=3$
and $n_{3}=2$ then $w_{\beta_{2}'}=w_{\gamma}$. It is easy to see
then that we get 4 isolated vertices corresponding to the original
$v_{1},w_{\beta_{1}},w_{\beta_{2}}$ and $v_{k}$. And the remaining
chains of open strings reflect removing $v_{1}$, $w_{\alpha}$, $w_{\beta_{1}'}$
and $w_{\beta_{1}}$ from the chain of length $n_{2}+1$ so the result
is a chain of length $n_{2}-3$ when $n_{2}>2$, or $0$ when $n_{2}=2$.
Similarly, after removing $v_{k}$, $w_{\beta_{2}}$, $w_{\beta_{2}'}$
and $w_{\gamma}$ from a chain of length $k-2+n_{3}+\sum_{i=4}^{k}n_{i}$
results a chain of length $k-3+n_{3}-2+n_{k}-1+\sum_{i=4}^{k-1}n_{i}$
when $k>3$, or $n_{3}-3$ when $k=3$ and $n_{3}>2$, or $0$ when
$k=3$ and $n_{3}=2$. 
\end{proof}

\begin{lemma}
\label{SubStinky}We adapt  \Cref{stringwithEdge}. Assume $w_{1}=v_{1}$ and $w_{\mu}=v_{k}$, i.e. $n_1=n_{k+1}=0$, and both $w_{1}=v_{1}$ and $w_{\mu}=v_{k}$ are open vertices. If $n_{i}=2+3l_{i}$ for $i=2,...,k$. Then the projective dimension
of $\mathcal{H}$ is $2(k-1)+2\sum_{i=2}^{k}l_{i}+1$.
\end{lemma}
 
\begin{proof}
If $k=2$, then $\mathcal{H}$ is an open cycle of length $2+n_2$, then by \Cref{LaundryList1} \ref{pdFormulaCycle}, $\pd(\mathcal{H})=2+n_2-1-\left\lfloor \frac{2+n_{2}-2}{3}\right\rfloor=2+2l_2+1 $. We now assume $k>2$. By making vertices $v_{1},...,v_{k}$ become closed, we obtain a hypergraph
$\mathcal{H}'$ and $\pd(\mathcal{H})\leq\pd(\mathcal{H}')$ by \Cref{LaundryList1} \ref{LMa2-2.9c}. Let
$\mathcal{H}''$ be the hypergraph obtained from $\mathcal{H}'$ by removing the higher
dimensional edge.  Then by \Cref{LaundryList1} \ref{removeMeets}, we have that $\pd(\mathcal{H}')=\pd(\mathcal{H}'')$ because all the vertices
$v_{1},...,v_{k}$ are closed in $\mathcal{H}'$.  Note that $\mathcal{H}''$
is a string of length $k+\sum_{i=2}^{k}n_{i}$ with $k-1$ open
strings of $n_{2},...,n_{k}$ open vertices. By Theorem 3.4 in
\cite{LMa1}, the projective dimension is the sum of projective dimension of each open string plus 1, hence we get $$\pd(\mathcal{H}'')=\sum_{i=2}^{k}n_{i}-\sum_{i=2}^{k}\left\lfloor \frac{n_{i}}{3}\right\rfloor +1=2(k-1)+2\sum_{i=2}^{k}l_{i}+1.$$
Thus, we obtain $\pd (\mathcal{H})\leq2(k-1)+2\sum_{i=2}^{k}l_{i}+1$. 

Now we consider $\mathcal{H}=\{v_{2}\}\cup \mathcal{H}_{v_{2}}$ and we will show it satisfies
the condition of Remark \ref{SplitH} with $V_{1}=\{v_{2}\}$ and
$V_{2}$ as the vertex set of $\mathcal{H}_{v_{2}}$. Denote $I_i$ as
the ideal generated by the generators of $I(\mathcal{H})$
corresponding to $V_i$, and $G_1=\{v_{2}\}=\mathcal{H}(I_1)$ and  $G_2=\mathcal{H}_{v_{2}}=\mathcal{H}(I_2)$. First notice that $\pd (G_{1})=1$ and $\reg (G_{1})=2$,
since the degree of the generator corresponding to $v_{2}$ is $3$. Moreover $G_2$ satisfies the
condition of \Cref{2Stringedge}, hence $\pd(G_{2})=2(k-1)+2\sum_{i=2}^{k}l_{i}=q$
and $\reg(G_{2})\leq k-1+\sum_{i=2}^{k}l_{i}=r-1$. By Lemma
\ref{IntersectionSplit}, 
$I_{1}\cap I_{2}=m_{v_2}I'$ where $\mathcal{H}(I')$ has $4$ isolated vertices
and two strings of length $n_{2}-3$ and
$k-3+n_{k}-1+n_{3}-2+\sum_{i=4}^{k-1}n_{i}$.
Then by $n_{i}=2+3l_{i}$ for $i=2,...,k$, \Cref{LaundryList1} \ref{pdFormulaString}, and \ref{RegFormula} we get
\begin{align*}
\begin{split}
\pd(\mathcal{H}(I_{1}\cap I_{2})) {}& =4+n_{2}-3-\left\lfloor
                             \frac{n_{2}-3}{3}\right\rfloor
                             + k- 6+\sum_{i=3}^{k}n_{i} 
                             -\left\lfloor
                               \frac{k-6+\sum_{i=3}^{k}n_{i}}{3}\right\rfloor\\
                             {}&=3+3l_{2}-\left\lfloor
                             \frac{3l_{2}-1}{3}\right\rfloor
                             + k- 6+\sum_{i=3}^{k}(3l_{i}+2) 
                             -\left\lfloor
                               \frac{k-6+\sum_{i=3}^{k}(3l_{i}+2)}{3}\right\rfloor\\
                               {}&=2l_{2}-2
                             + k+2(k-2)+\sum_{i=3}^{k}(3l_{i})
                             -\left\lfloor
                               \frac{k-6+2(k-2)+\sum_{i=3}^{k}(3l_{i})}{3}\right\rfloor\\
                               {}&=2l_{2}-2
                             + k+2(k-2)+\sum_{i=3}^{k}(3l_{i}) 
                             -k+4-\sum_{i=3}^{k}(l_{i})
                             \end{split}\\
 {}& =2(k-1)+2\sum_{i=2}^{k}l_{i}=q
\end{align*}
and 
\begin{align*}
\reg(\mathcal{H}(I_{1}\cap I_{2})) &= 3+\left\lceil
                             \frac{n_{2}-3}{3}\right\rceil
                             +\left\lceil
                             \frac{k-6+\sum_{i=3}^{k}n_{i}}{3}\right\rceil
  \\
  &= 3+\left\lceil
                             \frac{3l_{2}-1}{3}\right\rceil
                             +\left\lceil
                             \frac{k-6+2(k-2)+\sum_{i=3}^{k}(3l_{i})}{3}\right\rceil
  \\
&=3+l_{2}+k-3+\sum_{i=3}^{k}l_{i}=k+\sum_{i=2}^{k}l_{i}=r.
\end{align*}

Moreover, by \Cref{LaundryList1} \ref{pdFormulaString} and \ref{BettiString2}, we have
$\beta_{q,q+r} (\mathcal{H}(I_{1}\cap I_{2}))\neq 0$.  Hence by Lemma \ref{Split}, $\pd(\mathcal{H})=2(k-1)+2\sum_{i=2}^{k}l_{i}+1$. For the smaller cases, one can check similarly. 
\end{proof}

The following example offers a view of the "splitting" in the proof of \Cref{SubStinky}.

\begin{example} \label{ex:substinky}
Let $\mathcal{H}$ be the whole hypergraph in the left side of \Cref{HGStinky}. Let $V_1$ be
the $w_4=v_2$  (black part) of $\mathcal{H}$ and $V_2=\{w_1, w_2, w_3, w_5,
w_6, w_7,w_8,w_9,w_{10}\}$ (blue part) of $\mathcal{H}$. Let $G_1$ and $G_2$ be the hypergraphs associated to the vertex sets $V_1$ and $V_2$. Then the edges $\{w_3, w_4\}$, $\{w_4, w_5\}$, and $\{w_1, w_4, w_7,w_{10} \}$  (the purple part) are the shared
edges or edges of $G_1$ and $G_2$. In the right of \Cref{HGStinky}, we show
the hypergraphs for
$G_1$ and $G_2$ separately.

\begin{figure}[h] 
\caption{}\label{HGStinky}
\begin{center}
\begin{tikzpicture}[thick, scale=0.8]
\draw  [shape=circle, color=blue](0,-0.5) circle (.1) node[left]{$w_1 = v_1$} ;
\draw  [shape=circle, color=blue] (1,-1) circle (.1)  node[below]{$w_2$};
\draw  [shape=circle, color=blue] (2,-1.5) circle (.1) node[below]{$w_3$} ;
\draw  [shape=circle] (3,-1) circle (.1) node[right]{$w_4 = v_2$} ;
\draw  [shape=circle, color=blue] (4,-0.5) circle (.1) node[right]{$w_5$} ;
\draw  [shape=circle, color=blue] (5,0) circle (.1) node[right]{$w_6$} ;
\draw  [shape=circle, color=blue] (4,0.5) circle (.1) node[right]{$w_7= v_3$} ;
\draw  [shape=circle, color=blue] (3,1) circle (.1) node[above]{$w_8$} ;
\draw  [shape=circle, color=blue] (2,1.5) circle (.1) node[above]{$w_{9}$} ;
\draw [shape=circle, color=blue]  (1,1) circle (.1) node[left]{$w_{10} = v_4$};

\draw [line width=1.2pt, color=blue ] (0,-0.5)--(1,-1);
\draw [line width=1.2pt, color=blue ] (1,-1)--(2,-1.5);
\draw [line width=1.2pt, color=purple ] (2,-1.5)--(3,-1);
\draw [line width=1.2pt, color=purple ] (3,-1)--(4,-0.5);
\draw [line width=1.2pt, color=blue ] (4,-0.5)--(5,0);
\draw [line width=1.2pt, color=blue ] (5,0)--(4,0.5);
\draw [line width=1.2pt, color=blue ] (4,0.5)--(3,1);
\draw [line width=1.2pt, color=blue ] (3,1)--(2,1.5);
\draw [line width=1.2pt, color=blue ] (2,1.5)--(1,1);
\path [pattern=north west lines, pattern  color=purple]   (0,-0.5)--(3,-1)--(4,0.5)--(1,1)--cycle;
\path (2,-2)--(3,-2) node [pos=.5, below ] {$\mathcal{H}$ };

\draw  [shape=circle, color=blue](8,-0.5) circle (.1) node[left]{$w_1 = v_1$} ;
\draw  [shape=circle, color=blue] (9,-1) circle (.1)  node[below]{$w_2$};
\draw [shading=ball, ball color=blue] (10,-1.5) circle (.1) node[below]{$w_3$} ;
\draw [shading=ball, ball color=black] (11,-1) circle (.1) node[right]{$G_1=\{v_2\}$} ;
\draw  [shading=ball, ball color=blue](12,-0.5) circle (.1) node[right]{$w_5$} ;
\draw  [shape=circle, color=blue] (13,0) circle (.1) node[right]{$w_6$} ;
\draw  [shape=circle, color=blue] (12,0.5) circle (.1) node[right]{$w_7= v_3$} ;
\draw  [shape=circle, color=blue] (11,1) circle (.1) node[above]{$w_8$} ;
\draw  [shape=circle, color=blue] (10,1.5) circle (.1) node[above]{$w_{9}$} ;
\draw [shape=circle, color=blue]  (9,1) circle (.1)
node[left]{$w_{10} = v_4$};
\path (7,0.5)--(8,0.5) node [pos=.5, below ] {\textcolor{blue}{$G_2=\mathcal{H}_{v_2}$} };

\draw [line width=1.2pt, color=blue ] (8,-0.5)--(9,-1);
\draw [line width=1.2pt, color=blue ] (9,-1)--(10,-1.5);
\draw [line width=1.2pt, color=blue ] (12,-0.5)--(13,0);
\draw [line width=1.2pt, color=blue ] (13,0)--(12,0.5);
\draw [line width=1.2pt, color=blue ] (12,0.5)--(11,1);
\draw [line width=1.2pt, color=blue ] (11,1)--(10,1.5);
\draw [line width=1.2pt, color=blue ] (10,1.5)--(9,1);
\path [pattern=north west lines, pattern  color=purple]   (8,-0.5)--(12,0.5)--(9,1)--cycle;

\end{tikzpicture}
\end{center}
\end{figure}

\end{example}

Now with \Cref{SubStinky} we are ready to
address the case when the sum of the $n_i$ is $2k$ modulo 3.  In this
case \Cref{StinkyString2k} will be an instance of the special sub-case,
and \Cref{String2k} will give the general result.

\begin{lemma}
\label{StinkyString2k}We adapt  \Cref{stringwithEdge}. If $r_{1}=1=r_{k+1}$,
and $r_{i}=2$ for all $1<i<k+1$, then $\pd(\mathcal{H})=\pd(\mathcal{H}_{S_{\mu}})+1$.
\end{lemma}

\begin{proof}
First notice that $\pd(\mathcal{H}_{S_{\mu}})=k+\sum_{i=1}^{k+1}n_{i}-\left\lfloor \frac{k+\sum_{i=1}^{k+1}n_{i}}{3}\right\rfloor $ and $\sum_{i=1}^{k+1}r_{i}=2k$. We simplify $\pd(\mathcal{H}_{S_{\mu}})$ first.
\begin{align*}
\begin{split}
\pd(\mathcal{H}_{S_{\mu}})&=k+\sum_{i=1}^{k+1}n_{i}-\left\lfloor \frac{k+\sum_{i=1}^{k+1}n_{i}}{3}\right\rfloor\\
                             &=k+2k+\sum_{i=1}^{k+1}(3l_{i})-\left\lfloor \frac{k+2k+\sum_{i=1}^{k+1}(3l_{i})}{3}\right\rfloor
                            \end{split}\\
{} & =2k+\sum_{i=1}^{k+1}2l_{i}.
\end{align*}
Let $x_E$ be the variable corresponding to the edge $E_x$ that connecting
$v_{1}=w_{i_1}$ and the vertex of $w_{i_1-1}$ and $y_E$ be the variable corresponding
to the edge $E_y$ that connecting $v_{k}=w_{i_k}$ and the vertex of $w_{i_k+1}$.
We consider the short exact sequences: 
\[
0\rightarrow (\mathcal{H}:E_x) \rightarrow \mathcal{H}\rightarrow(\mathcal{H},x_E) \rightarrow0,
\]
and
\[
0\rightarrow ((\mathcal{H}:E_x):E_y) \rightarrow(\mathcal{H}:E_x)\rightarrow((\mathcal{H}:E_x),y_E)\rightarrow0.
\]

Notice that $((\mathcal{H}:E_x):E_y)$ is a union of two isolated vertices, two
strings of length $n_{1}-2$ and $n_{k+1}-2$, and a hypergraph that satisfies
assumptions of Lemma \ref{SubStinky}. Now with assumptions of $r_i$'s, we have 
\begin{align*}
\begin{split}
\pd((\mathcal{H}:E_x):E_y) &=2+n_{1}-2 -\left\lfloor
                                  \frac{n_{1}-2}{3}\right\rfloor
                                  +n_{k+1}-2-\left\lfloor
                                  \frac{n_{k+1}-2}{3}\right\rfloor \\
                            & +2(k-1)+2\sum_{i=2}^{k}l_{i}+1 \end{split}\\
                             & =2+2l_1+2_{l_{k+1}}+2(k-1)+2\sum_{i=2}^{k}l_{i}+1\\
 & =2k+\sum_{i=1}^{k+1}2l_{i}+1 =\pd(\mathcal{H}_{S_{\mu}})+1.
\end{align*}

Since $(\mathcal{H},x_E)$ is a union of an isolated vertex, a string of length
$n_{1}-1$, and a string of length $k-1+\sum_{i=2}^{k+1}n_{i}$ with
a $k-2$-dimensional edge such that $n_{2},...,n_{k+1}$ are the numbers
of vertices between vertices of $F$ and $\sum_{i=2}^{k+1}r_{i}=2k-1$.
By \Cref{String2k+1} and \Cref{LaundryList1} \ref{pdFormulaString}, 
\begin{align*}
\pd(\mathcal{H},x_E) & =1+n_{1}-1-\left\lfloor \frac{n_{1}-1}{3}\right\rfloor +k-1+\sum_{i=2}^{k+1}n_{i}-\left\lfloor \frac{k-1+\sum_{i=2}^{k+1}n_{i}}{3}\right\rfloor \\
& =2l_{1} +k+2k-1+\sum_{i=2}^{k+1}(3l_{i})-\left\lfloor \frac{k-1+2k-1+\sum_{i=2}^{k+1}(3l_{i})}{3}\right\rfloor \\
 & <2k+\sum_{i=1}^{k+1}2l_{i}+1 .
\end{align*}

Moreover, $((\mathcal{H}:E_x),y_E)$ is a union of two isolated vertices, two strings
of length $n_{1}-2$ and $n_{k+1}-1$ and a hypergraph satisfies the
assumptions of Lemma \ref{Stringedge} with a $k-2$-dimensional edge. Then by
Lemma \ref{Stringedge} and \Cref{LaundryList1} \ref{pdFormulaString}, 
\begin{align*}
\pd((\mathcal{H}:E_x),y_E) & =2+n_{1}-2-\left\lfloor \frac{n_{1}-2}{3}\right\rfloor +n_{k+1}-1-\left\lfloor \frac{n_{k+1}-1}{3}\right\rfloor \\
 & +k-1+\sum_{i=2}^{k}n_{i}-\left\lfloor \frac{k-1+\sum_{i=2}^{k}n_{i}}{3}\right\rfloor \\
 & =1+2l_1 +2l_{k+1}+k+2(k-1)+\sum_{i=2}^{k}(3l_{i})-\left\lfloor \frac{k-1+2k-2+\sum_{i=2}^{k}(3l_{i})}{3}\right\rfloor\\
 & <2k+\sum_{i=1}^{k+1}2l_{i}+1  =\pd((\mathcal{H}:E_x):E_y).
\end{align*}

Since $\pd((\mathcal{H}:E_x):E_y)>\pd((\mathcal{H}:E_x),y_E)$, we have 
\begin{align*}
\pd(\mathcal{H}:E_x) & \leq\max\{\pd((\mathcal{H}:E_x):E_y),\pd((\mathcal{H}:E_x),y_E)\}=\pd((\mathcal{H}:E_x):E_y)
\end{align*}
and $\pd((\mathcal{H}:E_x):E_y)\leq\max\{\pd(\mathcal{H}:E_x),\pd((\mathcal{H}:E_x),y_E)-1\}$.
This shows $\pd(\mathcal{H}:E_x)=\pd((\mathcal{H}:E_x):E_y)$. Similarly $\pd(\mathcal{H},x_E)<\pd(\mathcal{H}:E_x)$
gives $\pd(\mathcal{H})=\pd(\mathcal{H}:E_x)=\pd\mathcal{H}_{S_{\mu}}+1$.
\end{proof}

We are now finally ready for the case,  $\sum_{i=1}^{k+1}r_{i}=2k$.

\begin{proposition}
\label{String2k}We adapt  \Cref{stringwithEdge} and assume $\sum_{i=1}^{k+1}r_{i}=2k$
with $r_{1}=2$. If $r_{i}\neq0$ for all $i$, then $\pd(\mathcal{H})=(\pd\mathcal{H}_{S_{\mu}})+1$,
otherwise $\pd(\mathcal{H})=\pd(\mathcal{H}_{S_{\mu}})$.

\end{proposition}

\begin{proof}
We consider the short exact sequence 
\[
0\rightarrow \mathcal{Q}_{v_{1}} \rightarrow \mathcal{H}_{v_{1}}\rightarrow \mathcal{H} \rightarrow0.
\]
 We will show
$\pd(\mathcal{H}_{v_{1}})>\pd(\mathcal{Q}_{v_1})$. 
Moreover, we will show that when $r_{i}\neq0$ for all $i$, then $\pd(\mathcal{H}_{v_{1}})=\pd(\mathcal{H}_{S_{\mu}})+1$,
and otherwise
$\pd(\mathcal{H}_{v_{1}})=\pd(\mathcal{H}_{S_{\mu}})$. These
two claims with \Cref{shortexact} will prove the proposition.  We proceed by induction
on $k$ for both claims. 

We observe that $\mu=k+\sum_{i=1}^{k+1}n_{i}$ and $\pd(\mathcal{H}_{S_{\mu}})=2k+\sum_{i=1}^{k+1}2l_{i}$ as in the proof of \Cref{StinkyString2k}.

When $k=2$, observe that by Definition 2.6 and Discussion 2.8 in \cite{LMa2}, $\mathcal{Q}_{v_{1}}$ is a union of two isolated
vertices, and two strings of length $n_{1}-2$ and $n_{2}-2+n_{3}+1$.
Hence by \Cref{LaundryList1} \ref{pdFormulaString}, and $r_1=2$ implies $r_2+r_3=2$, 
\begin{align*}\pd(\mathcal{Q}_{v_{1}})& =2+n_{1}-2-\left\lfloor
    \frac{n_{1}-2}{3}\right\rfloor +n_{2}+n_{3}-1-\left\lfloor
    \frac{n_{2}+n_{3}-1}{3}\right\rfloor \\
    & =2+2l_1+2l_2+2l_3+1 \\
  & <4+\sum_{i=1}^{4}2l_{i}=\pd(\mathcal{H}_{S_{\mu}})
\end{align*}
Notice $\mathcal{H}_{v_{1}}$ is a union of a string of length $n_{1}$ and
a string of length $n_2+n_{3}+1$ with two open strings having $n_{2}-1$
and $n_{3}-1$ open vertices. By \Cref{LaundryList1} \ref{BettiString2} and Theorem 3.4 in \cite{LMa1} with $r_2+r_3=2$, we get 
\begin{align*}\pd(\mathcal{H}_{v_{1}})&=n_1-\left\lfloor \frac{n_{1}}{3}\right\rfloor+n_2+n_3+1-(2+ \left\lfloor \frac{n_{2}-1-1}{3}\right\rfloor+\left\lfloor \frac{n_{3}-1-1}{3}\right\rfloor)\\
&=2+2l_1+2l_2+2l_3+3=\pd(\mathcal{H}_{S_{\mu}})+1.
\end{align*}
when $r_{i}\neq0$ for all $i$. When $r_{i}=0$ for some $i$, then \begin{align*}\pd(\mathcal{H}_{v_{1}})&=n_1-\left\lfloor \frac{n_{1}}{3}\right\rfloor+n_2+n_3+1-(2+ \left\lfloor \frac{n_{2}-1-1}{3}\right\rfloor+\left\lfloor \frac{n_{3}-1-1}{3}\right\rfloor)\\
&=2+2l_1+2l_2+2l_3+2=\pd(\mathcal{H}_{S_{\mu}}).
\end{align*}

In both cases, we have $\pd(\mathcal{H}_{v_{1}})>\pd(\mathcal{Q}_{v})$ and
this concludes the case when $k=2$.

Now suppose $F$ is a $k$-dimensional edge with $k+1$ vertices, $v_{1},...,v_{k+1}.$
Suppose $r_{i}\neq0$ for all $i$ then again by Definition 2.6 and
discussion 2.8 in \cite{LMa2}, $\mathcal{Q}_{v_{1}}$ is either
\begin{enumerate}
\item  a union of
two isolated vertices, a string of length $n_{1}-2$ and a string
of length $\sum_{i=2}^{k+2}n_{i}+k-2$ when $n_{2}\geq2$, or 
\item $\mathcal{Q}_{v_{1}}$
is a union of two isolated vertices, a string of length $n_{1}-2$
and a string of length $\sum_{i=3}^{k+2}n_{i}+k-1$ when $n_{2}=1$.
\end{enumerate}

For the later case, $\sum_{i=3}^{k+2}r_{i}=2k-1$, then by \Cref{LaundryList1} \ref{pdFormulaString}, \begin{align*}\pd(\mathcal{Q}_{v_1})&=2+n_{1}-2-\left\lfloor
                                                           \frac{n_{1}-2}{3}\right\rfloor
                                                           +k-1+\sum_{i=3}^{k+2}n_{i}-\left\lfloor
                                                           \frac{k-1+\sum_{i=3}^{k+2}n_{i}}{3}\right\rfloor\\
                                   &=2+2l_{1}+k-1+2k-1+\sum_{i=3}^{k+2}3l_{i}-\left\lfloor
                                                    \frac{k-1+2k-1+\sum_{i=3}^{k+2}3l_{i}}{3}\right\rfloor\\
                     &=2l_1+2k+\sum_{i=3}^{k+2}(2l_{i})+1\\
                       &<2k+2+\sum_{i=1}^{k+2}(2l_{i})+1.\\
                      \end{align*}
  
For the first case, by \Cref{LaundryList1} \ref{pdFormulaString}, we
have \begin{align*}\pd(\mathcal{Q}_{v_{1}})&=2+n_{1}-2-\left\lfloor
                                                 \frac{n_{1}-2}{3}\right\rfloor
                                                 +k-2+\sum_{i=2}^{k+2}n_{i}-\left\lfloor
                                                 \frac{k-2+\sum_{i=2}^{k+2}n_{i}}{3}\right\rfloor\\
                                        &=2l_{1}+k+2k+\sum_{i=2}^{k+2}3l_{i}-\left\lfloor
                                                 \frac{k-2+2k+\sum_{i=2}^{k+2}3l_{i}}{3}\right\rfloor\\
                                        &=2k+1+\sum_{i=1}^{k+2}2l_{i}\\
       &<2k+2+\sum_{i=1}^{k+2}2l_{i}+1 
       \end{align*}
with the fact $r_{1}+...+r_{k+2}=2k+2$ and $r_{1}=2$. 

Now in both cases consider that $\mathcal{H}_{v_{1}}$ is a union of a string of length $n_{1}$ and a
string of length $\sum_{i=2}^{k+2}n_{i}+k$ with a $k-1$-dimensional
edge. Notice that $\sum_{i=2}^{k+2}r_{i}=2k$ and $r_{i}\neq0$ for all
$i$ such that $2\leq i\leq k+2$. By induction and \Cref{LaundryList1} \ref{pdFormulaString},
we have \begin{align*}\pd(\mathcal{H}_{v_{1}})&=n_{1}-\left\lfloor
                                                    \frac{n_{1}}{3}\right\rfloor
                                                    +k+\sum_{i=2}^{k+2}n_{i}-\left\lfloor
                                                    \frac{k+\sum_{i=2}^{k+2}n_{i}}{3}\right\rfloor+1\\
                                            &=2k+2+\sum_{i=1}^{k+2}2l_{i}+1 =\pd(\mathcal{H}_{S_{\mu}})+1
                                                  \end{align*}
where we use the fact that $r_{2}+...+r_{k+2}=2k$ and $r_{1}=2$.
Therefore $\pd(\mathcal{Q}_{v_{1}})<\pd(\mathcal{H}_{v_{1}})$,
and
$\pd(\mathcal{H})=\pd(\mathcal{H}_{v_{1}})=\pd(\mathcal{H}_{S_{\mu}})+1$,
satisfying the 2 claims.

Now suppose $r_{i}=0$ for some $i>1$ then $r_{j}=2$ for all $j\neq i$.
Notice again by the discussions in \cite{LMa2} that
$\mathcal{Q}_{v_{1}}$ is either
\begin{enumerate}
\item  a union of two isolated vertices, a string
of length $n_{1}-2$ and a string of length $\sum_{i=2}^{k+2}n_{i}+k-2$
when $n_{2}\geq2$, or 
\item $\mathcal{Q}_{v_{1}}$ is a union of two isolated vertices,
a string of length $n_{1}-2$ and a string of length $\sum_{i=3}^{k+2}n_{i}+k-2$
when $n_{2}=0$. 
\end{enumerate}

For the later case, we have $\sum_{i=3}^{k+2}r_{i}=2k$.
By \Cref{LaundryList1} \ref{pdFormulaString}, 
\begin{align*}\pd(\mathcal{Q}_{v_{1}})&=2+n_{1}-2-\left\lfloor
  \frac{n_{1}-2}{3}\right\rfloor
  +k-2+\sum_{i=3}^{k+2}n_{i}-\left\lfloor
  \frac{k-2+\sum_{i=3}^{k+2}n_{i}}{3}\right\rfloor\\
  &=2l_1+k+2k
  +\sum_{i=3}^{k+2}3l_{i}-\left\lfloor
  \frac{k-2+2k+\sum_{i=3}^{k+2}3l_{i}}{3}\right\rfloor\\
  &=2k+1
  +\sum_{i=1}^{k+2}2l_{i}\\
  &<2k+2+\sum_{i=1}^{k+2}2l_{i}. \end{align*}
For the first case, by \Cref{LaundryList1} \ref{pdFormulaString}, we
have \begin{align*}\pd(\mathcal{Q}_{v_{1}})&=2+n_{1}-2-\left\lfloor
                                                 \frac{n_{1}-2}{3}\right\rfloor
                                                 +k-2+\sum_{i=2}^{k+2}n_{i}-\left\lfloor
                                                 \frac{k-2+\sum_{i=2}^{k+2}n_{i}}{3}\right\rfloor
       \\  &=2l_1+k+2k
  +\sum_{i=2}^{k+2}3l_{i}-\left\lfloor
  \frac{k-2+2k+\sum_{i=2}^{k+2}3l_{i}}{3}\right\rfloor\\
  &=2k+1
  +\sum_{i=1}^{k+2}2l_{i}\\
  &<2k+2+\sum_{i=1}^{k+2}2l_{i}. \end{align*}
with the fact that $r_{2}+...+r_{k+2}=2k$ and $r_{1}=2$. Similar to
the case where $r_i$ is never $0$, we get that  $\mathcal{H}_{v_{1}}$ is a union of a string of length $n_{1}$
and a string of length $\sum_{i=2}^{k+2}n_{i}+k$ with a $k-1$-dimensional
edge. Notice that $\sum_{i=2}^{k+2}r_{i}=2k$ and $r_{i}=0$ for some $2\leq i\leq k+2$.
So by induction and \Cref{LaundryList1} \ref{pdFormulaString}, we
have \begin{align*}\pd(\mathcal{H}_{v_{1}})&=n_{1}-\left\lfloor
                                                 \frac{n_{1}}{3}\right\rfloor
                                                 +k+\sum_{i=2}^{k+2}n_{i}-\left\lfloor
                                                 \frac{k+\sum_{i=2}^{k+2}n_{i}}{3}\right\rfloor
       \\  &=2k+2+\sum_{i=1}^{k+2}2l_{i}\\
       &=\pd (\mathcal{H}_{\mathcal{S}_{\mu}})  \end{align*}
where we use the fact that $r_{2}+...+r_{k+2}=2k$ and $r_{1}=2$.
Hence $\pd(\mathcal{Q}_{v_{1}})<\pd(\mathcal{H}_{v_{1}})$, and
$\pd(\mathcal{H})=\pd(\mathcal{H}_{v_{1}})=\pd(\mathcal{H}_{S_{\mu}})$,
thus finishing the proof. 
\end{proof}

Now tying together \Cref{String2k-1}, \Cref{String2k+1}, and
\Cref{String2k} we can prove the following result.

\begin{theorem}
\label{StringedgeAll}Let $\mathcal{H}_{S_{\mu}}$ be a string hypergraph such
that it has $\mu$ vertices with all vertices are open except the end
vertices. We adapt  \Cref{stringwithEdge}. If $r_{1}+...+r_{k+1}=2k$
and $r_{i}\neq0$ for all $i$, then $\pd(\mathcal{H})=\pd(\mathcal{H}_{S_{\mu}})+1$
otherwise $\pd(\mathcal{H})=\pd(\mathcal{H}_{S_{\mu}})=\mu-\left\lfloor \frac{\mu}{3}\right\rfloor $.

\end{theorem}
\begin{proof}
The final equality is coming from \Cref{LaundryList1} \ref{pdFormulaString}. For the cases when $r_{1}+...+r_{k+1}<2k$ or $r_{1}+...+r_{k+1}>2k$,
we have $\pd(\mathcal{H})=\pd(\mathcal{H}_{S_{\mu}})$ by Proposition \ref{String2k-1}
and Proposition \ref{String2k+1}. We are left to consider the case, $r_{1}+...+r_{k+1}=2k$. If $r_i=0$ for some $i$, then $r_j=2$ for all $j\neq i$, hence \Cref{String2k} applies. If $r_i \neq 0$ for all $i$, and $r_1=r_{k+1}=1$, then \Cref{StinkyString2k} applies, otherwise, we may assume $r_1=2$ and \Cref{String2k} applies again.
\end{proof}

\begin{example} \label{diffPd}
Three hypergraphs in \Cref{fig:diffpd} are strings attached with one extra blue edges (checkerboard region) such that $\mu = 9$, $k = 3$ and $r_{1}+...+r_{k+1}=2k=6$. $\mathcal{H}_1$ and $\mathcal{H}_2$ satisfy $r_{i}\neq0$ for all $i$, but $\mathcal{H}_3$ has $r_4=0$. Therefore $\pd(\mathcal{H}_1)=\pd(\mathcal{H}_2)=\pd(\mathcal{H}_{\mathcal{S}_{\mu}})+1=9-\left\lfloor
                                                 \frac{9}{3}\right\rfloor+1=7$, and $\pd(\mathcal{H}_3)=\pd(\mathcal{H}_{\mathcal{S}_{\mu}})=6$ by \Cref{StringedgeAll}.

\begin{figure}[h] 
\caption{}\label{fig:diffpd}
\begin{center}
\begin{tikzpicture}[thick, scale=0.7, every node/.style={scale=0.8}]
\shade [shading=ball, ball color=black]  (0,0) circle (.1) node[left]{$w_1$};
\draw  [shape=circle] (1,-1.2) circle (.1) node[left]{$w_2$} ;
\draw  [shape=circle] (2,-2) circle (.1)  node[left]{$w_3= v_1$};
\draw  [shape=circle] (3.7,-2.5) circle (.1) node[below]{$w_4$} ;
\draw  [shape=circle] (4.7,-1.5) circle (.1) node[right]{$w_5 = v_2$} ;
\draw  [shape=circle] (5,0) circle (.1) node[right]{$w_6$} ;
\draw  [shape=circle] (4,1.2) circle (.1) node[right]{$w_7$} ;
\draw  [shape=circle] (2.5,1.7) circle (.1) node[right]{$w_8= v_3$} ;
\draw  [shading=ball, ball color=black] (1,1.2) circle (.1) node[above]{$w_9$} ;

\draw  (3.7,-3.5) node[below]{$\mathcal{H}_1$} ;

\draw [line width=1.2pt] (1,-1.2)--(0,0);
\draw [line width=1.2pt] (1,-1.2)--(2,-2);
\draw [line width=1.2pt] (2,-2)--(3.7,-2.5);
\draw [line width=1.2pt] (3.7,-2.5)--(4.7,-1.5);
\draw [line width=1.2pt ] (4.7,-1.5)--(5,0);
\draw [line width=1.2pt] (5,0)--(4,1.2);
\draw [line width=1.2pt] (4,1.2)--(2.5,1.7);
\draw [line width=1.2pt] (2.5,1.7)--(1,1.2);

\path [pattern=checkerboard, pattern  color=blue]   (2,-2)--(4.7,-1.5)--(2.5,1.7)--cycle;

\shade [shading=ball, ball color=black]  (7,0) circle (.1) node[left]{$w_1$};
\draw  [shape=circle] (8,-1.2) circle (.1) node[left]{$w_2$} ;
\draw  [shape=circle] (9,-2) circle (.1)  node[left]{$w_3= v_1$};
\draw  [shape=circle] (10.7,-2.5) circle (.1) node[below]{$w_4$} ;
\draw  [shape=circle] (11.7,-1.5) circle (.1) node[right]{$w_5 = v_2$} ;
\draw  [shape=circle] (12,0) circle (.1) node[right]{$w_6$} ;
\draw  [shape=circle] (11,1.2) circle (.1) node[right]{$w_7= v_3$} ;
\draw  [shape=circle] (9.5,1.7) circle (.1) node[right]{$w_8$} ;
\draw  [shading=ball, ball color=black] (8,1.2) circle (.1) node[above]{$w_9$} ;

\draw  (10.7,-3.5) node[below]{$\mathcal{H}_2$} ;
\draw [line width=1.2pt] (8,-1.2)--(7,0);
\draw [line width=1.2pt] (8,-1.2)--(9,-2);
\draw [line width=1.2pt] (9,-2)--(10.7,-2.5);
\draw [line width=1.2pt] (10.7,-2.5)--(11.7,-1.5);
\draw [line width=1.2pt ] (11.7,-1.5)--(12,0);
\draw [line width=1.2pt] (12,0)--(11,1.2);
\draw [line width=1.2pt] (11,1.2)--(9.5,1.7);
\draw [line width=1.2pt] (9.5,1.7)--(8,1.2);

\path [pattern=checkerboard, pattern  color=blue]   (9,-2)--(11.7,-1.5)--(11,1.2)--cycle;

\shade [shading=ball, ball color=black]  (14,0) circle (.1) node[left]{$w_1$};
\draw  [shape=circle] (15,-1.2) circle (.1) node[left]{$w_2$} ;
\draw  [shape=circle] (16,-2) circle (.1)  node[left]{$w_3= v_1$};
\draw  [shape=circle] (17.7,-2.5) circle (.1) node[below]{$w_4$} ;
\draw  [shape=circle] (18.7,-1.5) circle (.1) node[right]{$w_5$} ;
\draw  [shape=circle] (19,0) circle (.1) node[right]{$w_6= v_2$} ;
\draw  [shape=circle] (18,1.2) circle (.1) node[right]{$w_7$} ;
\draw  [shape=circle] (16.5,1.7) circle (.1) node[right]{$w_8$} ;
\draw  [shading=ball, ball color=black] (15,1.2) circle (.1) node[above]{$w_9= v_3$} ;

\draw  (17.7,-3.5) node[below]{$\mathcal{H}_3$} ;
\draw [line width=1.2pt] (15,-1.2)--(14,0);
\draw [line width=1.2pt] (15,-1.2)--(16,-2);
\draw [line width=1.2pt] (16,-2)--(17.7,-2.5);
\draw [line width=1.2pt] (17.7,-2.5)--(18.7,-1.5);
\draw [line width=1.2pt ] (18.7,-1.5)--(19,0);
\draw [line width=1.2pt] (19,0)--(18,1.2);
\draw [line width=1.2pt] (18,1.2)--(16.5,1.7);
\draw [line width=1.2pt] (16.5,1.7)--(15,1.2);

\path [pattern=checkerboard, pattern  color=blue]   (16,-2)--(19,0)--(15,1.2)--cycle;

\end{tikzpicture}
\end{center}
\end{figure}
\end{example}

\begin{remark}
With \Cref{StringedgeAll}, one can easily compute the projective dimension of a string with extra edges. We first remove all the edges that are union of other edges using \Cref{LaundryList1} \ref{removeMeets}. If we are left with one higher dimensional edge, we can use \Cref{StringedgeAll} and \Cref{LaundryList1} \ref{pdFormulaString} to compute the projective dimension. \Cref{removeEdges} illustrates the process.
\end{remark}

\begin{example} \label{removeEdges}
Let $\mathcal{H}$ be the hypergraph  shown in \Cref{HGStringAll}. We can remove all red edges (dashed vertical line and the bricks regions) using \Cref{LaundryList1} \ref{removeMeets} and remove the blue edge (checkerboard regions) using \Cref{StringedgeAll}. The projective dimension of the hypergraph is $11-\left\lfloor \frac{11}{3}\right\rfloor=8$.

\begin{figure}[h] 
\caption{}\label{HGStringAll}
\begin{center}
\begin{tikzpicture}[thick, scale=0.8]
\draw [shading=ball, ball color=black]  (-2,0) circle (.1) node[left]{$v_1$};
\draw [shape=circle]  (-1,-0.5) circle (.1) node[below]{$v_2$};
\draw [shape=circle]  (0,-0.5) circle (.1) node[below]{$v_3$};
\draw  [shape=circle] (1,-0.5) circle (.1) node[below]{$v_4$} ;
\draw  [shape=circle] (2,0) circle (.1) node[below]{$v_5$} ;
\draw  [shape=circle] (2.5,0.5) circle (.1) node[right]{$v_6$} ;
\draw  [shape=circle] (2,1) circle (.1) node[above]{$v_7$} ;
\draw  [shape=circle] (1,1.5) circle (.1) node[above]{$v_8$} ;
\draw [shape=circle](0,1.5) circle (.1) node[above]{$v_9$};
\draw [shape=circle](-1,1.5) circle (.1) node[above]{$v_{10}$} ;
\draw [shading=ball, ball color=black](-2,1) circle (.1)  node[left]{$v_{11}$} ;
\draw [line width=1.2pt] (-1,-0.5)--(-2,0);
\draw [line width=1.2pt] (-1,-0.5)--(0,-0.5);
\draw [line width=1.2pt] (0,-0.5)--(1,-0.5);
\draw [line width=1.2pt] (1,-0.5)--(2,0);
\draw [line width=1.2pt] (2,0)--(2.5,0.5);
\draw [line width=1.2pt] (2.5,0.5)--(2,1) ;
\draw [line width=1.2pt] (2,1)--(1,1.5);
\draw [line width=1.2pt] (1,1.5)--(0,1.5);
\draw [line width=1.2pt] (-1,1.5)--(0,1.5);
\draw [line width=1.2pt] (-2,1)--(-1,1.5);
\draw [dashed, color=red] (-2,0)--(-2,1);
\path [pattern=bricks, pattern  color=red]   (2,0)--(2.5,0.5)--(2,1)--(1,1.5)--cycle;
\path [pattern=bricks, pattern  color=red]   (-1,-0.5)--(0,-0.5)--(0,1.5)--(-1,1.5)--cycle;
\path [pattern=checkerboard, pattern  color=blue]   (0,-0.5)--(1,-0.5)--(2,0)--(0,1.5)--cycle;
\path (-0.5,-1)--(1.5,-1) node [pos=.5, below ] {$\mathcal{H}$ };

\end{tikzpicture}
\end{center}
\end{figure}

\end{example}

\section{Cycles with higher dimensional edges}\label{CyclesSec}

Now we examine the case where we have added a higher dimensional
edge to an open cycle.

\begin{notation}\label{CyclewithEdge}\hangindent\leftmargini
\hskip\labelsep
Let $\mathcal{H}$ be a hypergraph and $F$ be an edge of $\mathcal{H}$. 
\begin{enumerate}[label={(\arabic*)},ref={(\arabic*)}]
\item Let $\mathcal{H}_{C_\mu}$ be an open cycle with vertex set $V=\{w_1, \dots,
w_\mu\}$, and an edge $F$ with $k>1$ vertices $\{v_1=w_{i_1}, \dots, v_k=w_{i_k}
\}\subseteq \{w_1, \dots,
w_\mu\}$ with $1\leq i_1< \dots <i_k \leq \mu$. When we say $\mathcal{H}$ is \emph{ an open cycle together with an edge consisting of $k$ vertices}, we mean $\mathcal{H}=\mathcal{H}_{C_\mu} \cup F$. 
\item Let $n_j=i_j-i_{j-1}-1$ for $j=1, \dots k$ be the number of vertices between vertices of $F$.
\item We write $n_i=3l_i+r_i$  where $l_{i}$ are some
non-negative integers, and $0\leq r_{i}\leq2$ for $i=1,...,k+1$.

\end{enumerate}


\end{notation}

We start by showing that the induced hypergraph of
$\mathcal{H}_{C_{\mu}}$ on the complement of $V_F$ has smaller
projective dimension than that of $\mathcal{H}_{C_{\mu}}$ when the sum
of (the $n_i$ modulo 3) is less than $2k-1$.  This is necessary in the proof of Theorem \ref{cycleF}.     

\begin{lemma}\label{CycleRemoveF}
We adapt \Cref{CyclewithEdge}. Let $\mathcal{H}_{V_{F}}=\mathcal{H}_{C_{\mu}}\cap(V\backslash V_{F})$ be the
hypergraph obtained by removing all the vertices $v_{1},...,v_{k}$.
If $r_{1}+...+r_{k}<2k-1$, then $\pd(\mathcal{H}_{V_{F}})<\pd\mathcal{H}_{C_{\mu}}$.
\end{lemma}

\begin{proof}
Notice that $\mu=k+\sum_{i=1}^{k} n_{i}$
and \begin{align*}\pd\mathcal{H}_{C_{\mu}}&=\mu-1-\left\lfloor
                                                  \frac{\mu-2}{3}\right\rfloor\\
      &=k-1+\sum_{i=1}^{k}n_{i}-\left\lfloor
        \frac{k-2+\sum_{i=1}^{k}n_{i}}{3}\right\rfloor\\
        &>k-1+\sum_{i=1}^{k}n_{i}-\left\lfloor
        \frac{k-2+2k-1+\sum_{i=1}^{k}3l_{i}}{3}\right\rfloor\\
        &=\sum_{i=1}^{k}n_{i}-\sum_{i=1}^{k}l_{i}\end{align*} by
    \Cref{LaundryList1} \ref{pdFormulaCycle} and the assumption $r_{1}+...+r_{k}<2k-1$. On the other hand, $\pd\mathcal{H}_{V_F}=\sum_{i=1}^{k}(n_{i}-\left\lfloor \frac{n_{i}}{3}\right\rfloor )=\sum_{i=1}^{k}n_{i}-\sum_{i=1}^{k}l_{i}$
because $\mathcal{H}_{V_F}$ is union of $k$ strings such that each string has
$n_{i}$ vertices. Notice that this notation allows that $n_{i}$ can be
$0$ for
some $i$. Hence we have $\pd(\mathcal{H}_{V_{F}})<\pd\mathcal{H}_{C_{\mu}}$.

\end{proof}

Next, we show that when the sum of the $n_i$ modulo 3 is greater than
$2k-1$ that the projective dimension of $\mathcal{H}$ is the same as
for the underlying cycle. 

\begin{lemma}\label{CycleLarge}
We adapt \Cref{CyclewithEdge}. If $\sum_{i=1}^{k}r_{i}\geq2k-1$, then
$\pd(\mathcal{H})=\pd(\mathcal{H}_{C_{\mu}})$.
\end{lemma}

\begin{proof}
By \Cref{LaundryList1} \ref{pdFormulaCycle} and the assumption $2k\geq \sum_{i=1}^{k}r_{i}\geq2k-1$, we have
\begin{align*}
\pd(\mathcal{H}_{C_{\mu}}) & =k-1+\sum_{i=1}^{k}n_{i}-\left\lfloor \frac{k-2+\sum_{i=1}^{n}n_{i}}{3}\right\rfloor \\
& =k-1+\sum_{i=1}^{k}n_{i}-\left\lfloor \frac{k-2+\sum_{i=1}^{k}r_{i}+\sum_{i=1}^{k}3l_{i}}{3}\right\rfloor \\
 & =\sum_{i=1}^{k}n_{i}-\sum_{i=1}^{k}l_{i}.
\end{align*}
Since $\sum_{i=1}^{k}r_{i}\geq2k-1$,
$r_{i}<3$, and $k>1$, we have at most one $r_{i}$ such that $r_{i}=1$.
We may assume $r_{k}=1$ if there is one otherwise $r_{i}=2$ for all $i$. 
Let $\mathcal{H}_{v_{1}}=\mathcal{H}_{v}$ be the hypergraph removing the vertex $v_{1}=v$ from $\mathcal{H}$
and let $\mathcal{H}_{v}:v_{1}=\mathcal{Q}_{v}$. We have a short exact sequence 
\[
0\rightarrow\mathcal{Q}_{v}\rightarrow\mathcal{H}_{v}\rightarrow\mathcal{H}\rightarrow0.
\]
 We will show that $\pd(\mathcal{H}_{v}) >\pd(\mathcal{Q}_{v})$, then by \Cref{shortexact},
we conclude that $\pd(\mathcal{H})=\pd(\mathcal{H}_{v})=\pd(\mathcal{H}_{C_{\mu}})$. 

When $k=2$, $\mathcal{H}_{v}$ is a string of length $n_1+n_2+1$ with two open strings of $n_1-1$ and $n_2-1$ open vertices. $\mathcal{Q}_{v}$ is the union of two closed vertices and an open string of $n_1-2+n_2-2+1$ vertices. By \Cref{LaundryList1} \ref{pdFormulaString} and Theorem 3.4 in \cite{LMa1}, we have $$\pd(\mathcal{H}_{v})=n_1+n_2+1-(2+\left\lfloor \frac{n_{1}-2}{3}\right\rfloor+\left\lfloor \frac{n_{2}-2}{3}\right\rfloor)+1=2l_1+2l_2+4,$$
and $$\pd(\mathcal{Q}_{v})=2+n_1+n_2-3-\left\lfloor \frac{n_{1}+n_2-3}{3}\right\rfloor=2l_1+2l_2+3$$ when $r_1+r_2=4$, and $$\pd(\mathcal{H}_{v})=n_1+n_2+1-(2+\left\lfloor \frac{n_{1}-2}{3}\right\rfloor+\left\lfloor \frac{n_{2}-2}{3}\right\rfloor)+1=2l_1+2l_2+3,$$ $$\pd(\mathcal{Q}_{v})=2+n_1+n_2-3-\left\lfloor \frac{n_{1}+n_2-3}{3}\right\rfloor=2l_1+2l_2+2$$ when $r_1+r_2=3$.

For $k>2$, we show $\pd(\mathcal{H}_{v})=k-1+\sum_{i=1}^{k}n_{i}-\left\lfloor \frac{k-1+\sum_{i=1}^{k}n_{i}}{3}\right\rfloor =\sum_{i=1}^{k}n_{i}-\sum_{i=1}^{k}l_{i} >\pd(\mathcal{Q}_{v})$. Notice that $\mathcal{H}_{v}$ is a string of length $\mu-1=k-1+\sum_{i=1}^{k}n_{i}$
attached with a $k-2$-dimensional edge of $k-1$ vertices. Moreover,
$\sum_{i=1}^{k}r_{i}\geq2k-1=2(k-1)+1$. By \Cref{String2k+1},
\begin{align*}
\pd(\mathcal{H}_{v}) & =\pd(\mathcal{H}_{S_{\mu-1}})=k-1+\sum_{i=1}^{k}n_{i}-\left\lfloor \frac{k-1+\sum_{i=1}^{k}n_{i}}{3}\right\rfloor  =\sum_{i=1}^{k}n_{i}-\sum_{i=1}^{k}l_{i}.
\end{align*}
 On the other hand, $\mathcal{Q}_{v}$ is the union of two closed
vertices and an open string of length $n_{1}-2+k-3+\sum_{i=2}^{k}n_{i}$
hence by \Cref{LaundryList1} \ref{pdFormulaString}, 
\begin{align*}
\pd(\mathcal{Q}_{v}) & =2 +k-5+\sum_{i=1}^{k}n_{i}-\left\lfloor \frac{k-5+\sum_{i=1}^{k}n_{i}}{3}\right\rfloor \\
 & \leq k-3+\sum_{i=1}^{k}n_{i}-\left\lfloor \frac{k-5+2k-1+\sum_{i=1}^{k}3l_{i}}{3}\right\rfloor \\
 & =\sum_{i=1}^{k}n_{i}-\sum_{i=1}^{k}l_{i} -1 <\pd(\mathcal{H}_{v})
\end{align*}
 when $\sum_{i=1}^{k}r_{i}\geq2k-1$. 
 \end{proof}

Now we can show that the projective dimension is always preserved
$\mathcal{H}$ and $\mathcal{H}_{C_{\mu}}$ differ by a single higher dimensional edge $F$.  

\begin{theorem}\label{cycleF}
Let $\mathcal{H}_{C_{\mu}}$ be the cycle hypergraph with $\mu$ open vertices.
Let $\mathcal{H}=\mathcal{H}_{C_{\mu}}\cup F$ where $F$ is a $k-1$-dimensional edge with $k$ vertices.
Then $\pd(\mathcal{H})=\pd(\mathcal{H}_{C_{\mu}})=\mu-1-\left\lfloor \frac{\mu-2}{3}\right\rfloor $.
\end{theorem}

\begin{proof}
The second equality is coming from \Cref{LaundryList1} \ref{pdFormulaCycle}. 
Let $V_F = \{v_1, \dots, v_k\}$, which is a subset of the vertex set of
$\mathcal{H}_{C_{\mu}}$, be the vertex set of $F$.  Let $n_i =
3l_i + r_i$ be the spacing between the $v_i$ and defined as before.
If $\sum_{i=1}^k r_i \geq 2k -1$ then the theorem holds by 
\Cref{CycleLarge}.  So we need only consider the case when
$\sum_{i=1}^k r_i < 2k -1$.

When $\sum_{i=1}^k r_i < 2k -1$, we will use \Cref{CycleRemoveF} and the short exact sequence 
\[0\rightarrow (\mathcal{H}:F)=\mathcal{H}_{C_{\mu}} \rightarrow\mathcal{H}\rightarrow (\mathcal{H},x_{F}) \rightarrow0\] where $x_F$ is the variable corresponding to the edge $F$.  The short exact sequence gives $\pd (\mathcal{H}) \leq \max \{ \pd (\mathcal{H}_{C_{\mu}}) , \pd (\mathcal{H},x_{F})\}$. Since $(\mathcal{H},x_{F})$ is the union of $\mathcal{H}_{V_F}$ and an isolated vertex presenting the variable of $x_F$, we have $\pd (\mathcal{H},x_{F})=\pd \mathcal{H}_{V_F}+1\leq \pd \mathcal{H}_{C_{\mu}} $ by \Cref{CycleRemoveF}. By \Cref{LaundryList1} \ref{LMa2-2.9c}, we have to $\pd \mathcal{H}_{C_{\mu}} \leq \pd \mathcal{H}$. Therefore we have
 $$\pd \mathcal{H} \leq \max \{ \pd (\mathcal{H}_{C_{\mu}}) , \pd (\mathcal{H},x_{F})\}=\pd \mathcal{H}_{C_{\mu}} \leq \pd \mathcal{H}.$$
 \end{proof}

\begin{remark}
As before, one can compute the projective dimension of an open cycle with extra edges either using  \Cref{LaundryList1} \ref{removeMeets} or above theorem. 
\end{remark}

\begin{remark}
It is natural to conjecture that given an open cycle, no matter how many higher dimensional edges are on the cycle the projective dimension of the hypergraph is the same as the projective dimension of the open cycle. There are more cases that one needs to consider for the proof of the conjecture. In particular, in \Cref{StringedgeAll} we have a case for strings where the projective dimension does not stay the same once a higher dimensional edge is removed. The proof of the case when the projective dimension jumped up by one requires a lot of steps. In light of this, new tools must be developed in order to prove the conjecture for cycles.  
\end{remark}

\end{document}